\documentclass{amsart}

\usepackage[utf8]{inputenc}

\usepackage{geometry}
\usepackage{hyperref}
\hypersetup{
     colorlinks = true,
     linkcolor = magenta,
     anchorcolor = magenta,
     citecolor = magenta,
     filecolor = magenta,
     urlcolor = magenta
}
     
\usepackage{graphicx}
\usepackage{amssymb}
\usepackage{amsthm}
\usepackage{mathrsfs}

\usepackage{stackrel}
\usepackage[all]{xy}

\usepackage[dvipsnames]{xcolor}

\newcommand{\Z}{ {\mathbf{Z}} } 
\newcommand{\Q}{ {\mathbf{Q}} }

\newcommand{\barQ}{ {\bar{\mathbf{Q}}} }
\newcommand{\C}{ {\mathbf{C}} }
\newcommand{\Zp}{ {\Z_p} }
\newcommand{\Qp}{ {\mathbf{Q}_p} }
\newcommand{\barQp}{ {\bar{\mathbf{Q}}_p} }
\newcommand{\Cp}{ \mathbf{C}_p }

\newcommand{\cO}{ {\mathcal{O}} }

\newcommand{\cU}{{\mathcal{U}}}

\usepackage{yfonts}

\renewcommand{\d}{\mathfrak d}

\newcommand{\cl}{\mathrm{cl}}
\newcommand{\Ad}{{ \mathrm{Ad}}}
\newcommand{\Sym}{{ \mathrm{Sym}}}

\newcommand{\GL}{\mathrm{GL}}

\newcommand{\SL}{\mathrm{SL}}

\newcommand{\GSp}{\mathrm{GSp}}

\newcommand{\ord}{\mathrm{ord}}

\newcommand{\new}{\mathrm{new}}

\newcommand{\SK}{{\mathrm{SK}}}

\newcommand{\cont}{\mathrm{cont}}

\newcommand{\Hom}{{\mathrm{Hom}}}

\newcommand{\hf}{ {\mathbf{f}} }
\newcommand{\hg}{ {\mathbf{g}} }
\newcommand{\hh}{ {\mathbf{h}} }
\newcommand{\hF}{ {\mathbf{F}} }
\newcommand{\hSK}{{\mathbf{SK}}}

\newcommand{\eord}{e_\mathrm{ord} }

\newcommand{\Lp}{\mathcal L_p}

 \newcommand{\comm}[1]{}

\newtheorem{theorem}{Theorem}[section]
\newtheorem{definition}[theorem]{Definition}
\newtheorem{lemma}[theorem]{Lemma}

\newtheorem{remark}[theorem]{Remark}
\newtheorem{corollary}[theorem]{Corollary}

\newtheorem{proposition}[theorem]{Proposition}

\title{On \MakeLowercase{\texorpdfstring{$p$}{p}}-adic \texorpdfstring{$L$}{L}-functions for \texorpdfstring{$\GL(2)\times \GL(3)$}{GL(2)xGL(3)} via pullbacks of Saito--Kurokawa lifts}

\author{Daniele Casazza}
\address{University College Dublin, Belfield, Dublin 4, Ireland}
\email{casazza.daniele@gmail.com}

\author{Carlos de Vera-Piquero}
\address{Facultad de Ciencias, Universidad de Zaragoza, C/ Pedro Cerbuna 12, 50009, Zaragoza, Spain.}
\email{devera@unizar.es}

\date{2021}

\begin{document}

\begin{abstract}
We build a one-variable $p$-adic $L$-function attached to two Hida families of ordinary $p$-stabilised newforms $\hf$, $\hg$, interpolating the algebraic part of the central values of the complex $L$-series $L(f \otimes \Ad(g), s)$ when $f$ and $g$ range over the classical specialisations of $\hf$, $\hg$ on a suitable line of the weight space. The construction rests on two major results: an explicit formula for the relevant complex central $L$-values, and the existence of non-trivial $\Lambda$-adic Shintani liftings and Saito--Kurokawa liftings studied in a previous work by the authors. We also illustrate that, under an appropriate sign assumption, this $p$-adic $L$-function arises as a factor of a triple product $p$-adic $L$-function attached to $\hf$, $\hg$, and $\hg$.
\end{abstract}


\maketitle

\setcounter{tocdepth}{1}
\tableofcontents

\section{Introduction}

The first significant work towards the study of $p$-adic $L$-functions for $\GL(2)\times \GL(3)$ was achieved by C.-G. Schmidt in \cite{Schmidt}. By using the so-called modular symbol method, extending ideas due to Kazhdan and Mazur, Schmidt studied algebraicity and $p$-adic interpolation properties of the twists of special values for the $L$-series associated with Rankin--Selberg convolution of two fixed automorphic representations for $\GL(2)$ and $\GL(3)$. This lead to the construction of the corresponding $p$-adic $L$-function in the cyclotomic variable.

The algebraicity results and $p$-adic interpolation properties studied in op. cit. have been pursued and generalised to $\GL(n)\times \GL(n+1)$ by several authors along the last years (see, for example, \cite{KS13,RS08,Rag10,Sun17,Jan19}). Concerning the construction of $p$-adic $L$-functions for $\GL(n)\times\GL(n+1)$, we must mention the work of F. Januszewski in \cite{Jan15}, and more recently in \cite{Jan}, where $p$-adic $L$-functions for Hida families on $\GL(n)\times\GL(n+1)$ are achieved. As a by-product, Januszewski obtains strong non-vanishing results for central $L$-values of twisted Rankin--Selberg $L$-functions.

Within the general framework provided by the above mentioned works, the aim of the present paper is to construct a $p$-adic $L$-function (of a single weight variable) for $\GL(2) \times \GL(3)$, using explicit Hida families of ordinary modular forms. Our construction relies on very explicit Ichino-like formulae for central values of degree $6$ $L$-series proven in the recent years, which are expressed in terms of periods involving Saito--Kurokawa liftings that can be $p$-adically interpolated when varying the weight. A disadvantage of this route is that we are forced to impose some technical assumptions that are inherent to the method and to the available formulae. However, it is expected that most of these assumptions could be removed by pushing further some previous works. In contrast, the advantage of our approach is that it permits a closer link to the arithmetic of classical modular forms and their special $L$-values (for example, via the study of appropriate special cycles in the relevant algebraic varieties). Therefore, our construction provides a good complement to the works cited in the previous paragraph.

In order to describe our main results, let $N\geq 1$ be an odd, squarefree integer, and $k, \ell \geq 1$ two odd integers. Let $f \in S_{2k}^{\new}(N)$ and $g \in S_{\ell+1}^{\new}(N)$ be two normalised newforms of level $\Gamma_0(N)$ and weights $2k$ and $\ell+1$, respectively, and let $V_f$ and $V_g$ denote the compatible system of $p$-adic Galois representations associated with $f$ and $g$, respectively. Let also $\Ad(V_g):=\Ad^0(V_g)$ be the {\em adjoint} representation of $V_g$ (which, since $g$ has trivial nebentype character, is just a cyclotomic twist of the symmetric square representation $\Sym^2(V_g)$). Associated with this data, we may consider the complex $L$-series $L(f \otimes \Ad(g), s)$ attached to $V_f \otimes \Ad(V_g)$, defined by an Euler product 
\[
    L(f \otimes \Ad(g), s) = \prod_q L_{(q)}(f \otimes \Ad(g), q^{-s})^{-1}
\]
for $\mathrm{Re}(s) \gg 0$. If $q$ is a prime not dividing $N$, and we write $\alpha_q(f)$, $\beta_q(f)$ and $\alpha_q(g)$, $\beta_q(g)$ for the roots of the $q$-th Hecke polynomials
\[
P_{f,q}(T) := T^2 - a_q(f)T + q^{2k-1},\qquad  P_{g,q}(T) := T^2 - a_q(g) T + q^{\ell},
\]
of $f$ and $g$, respectively, then $L(f\otimes \Ad(g), T)_{(q)}$ is the degree $6$ Euler factor defined by the identity
\begin{align} \label{eqn:q-factor}
    L_{(q)}(f\otimes \Ad(g), T)  
    =  \bigl(  1 - &  \alpha_q(f)  \tfrac{\alpha_q(g)}{\beta_q(g)} T   \bigr)
    \bigl(  1 - \alpha_q(f) T    \bigr)
    \bigl(  1 - \alpha_q(f) \tfrac{\beta_q(g)}{\alpha_q(g)}T   \bigr) 
    \cdot \\
    &
     \bigl(  1 -  \beta_q(f) \tfrac{\alpha_q(g)}{\beta_q(g)} T  \bigr)
    \bigl(  1 - \beta_q(f) T    \bigr)
    \bigl(  1 - \beta_q(f) \tfrac{\beta_q(g)}{\alpha_q(g)} T   \bigr). \notag   
\end{align}
The completed $L$-function 
\[
\Lambda(f \otimes \Ad(g),s) := L_{\infty}(f\otimes \Ad(g),s) \cdot L(f\otimes \Ad(g),s), 
\]
where 
\[
L_{\infty}(f\otimes \Ad(g),s) = \begin{cases}
\Gamma_{\C}(s)\Gamma_{\C}(s-2\ell)\Gamma_{\C}(s-\ell)^2 & \text{if } \ell < k,\\
\Gamma_{\C}(s)\Gamma_{\C}(s+1-2k)\Gamma_{\C}(s-\ell)^2 & \text{if } \ell \geq k,
\end{cases}
\]
with $\Gamma_{\C}(s) := 2(2\pi)^{-s}\Gamma(s)$, admits analytic continuation to the whole complex plane and satisfies a functional equation relating the values $\Lambda(f\otimes \Ad(g),2k-s)$ and $\Lambda(f\otimes \Ad(g),s)$. The center of symmetry $s = k$ is always a critical point in the sense of Deligne, and the global sign $\varepsilon = \{\pm 1\}$ appearing in this functional equation can be written as a product of local signs, one for each rational place:
\[
\varepsilon = \prod_v \varepsilon_v, \qquad \varepsilon_v \in \{\pm 1\}.
\]
The local signs $\varepsilon_v$ are $+1$ away from $N\infty$, and hence this is actually a finite product. In addition, our assumption that both $f$ and $g$ are newforms of the same level $N$, with $N$ odd and squarefree, implies that $\varepsilon_q = +1$ for all primes $q \mid N$. Therefore, the sign in the functional equation for $\Lambda(f\otimes \Ad(g),s)$ is completely governed by the local sign at the archimedean place $v = \infty$, i.e. $\varepsilon = \varepsilon_{\infty}$. Finally, one can check that
\[
\varepsilon_{\infty} = \begin{cases} 
-1 & \text{if } \ell < k, \\
+ 1 & \text{if } \ell \geq k. 
\end{cases}
\]
In the second case, one expects the central value $\Lambda(f\otimes \Ad(g),k)$ to be generically non-zero, and it is natural to aim for a (two-variable) $p$-adic $L$-function $\mathcal L_p(\hf,\Ad(\hg))$ interpolating the central values $\Lambda(f\otimes \Ad(g),k)$, in the region $\ell \geq k$, when both $f$ and $g$ vary in Hida families $\hf$ and $\hg$ of ordinary $p$-stabilised newforms. With this in mind, the goal of this note is to construct a one-variable $p$-adic $L$-function $\mathcal L_p^{\circ}(\hf, \Ad(\hg))$ attached to Hida families $\hf$, $\hg$ interpolating those central values along the line $\ell = k$, which shall provide what it is expected to be the restriction of the two-variable one.

Assume from now on that $\ell = k$. Our construction relies on an explicit central value formula for $\Lambda(f\otimes\Ad(g),k)$ in terms of Petersson products involving Saito--Kurokawa lifts, firstly described by Ichino \cite{Ichino} and recently generalised in \cite{PaldVP} and \cite{Chen}. For our construction, we will rely on the formulation in \cite{PaldVP}:

\begin{theorem} \label{thm:PaldVP Introduction}
    Let $N\geq 1$ be a squarefree odd integer, $k\geq 1$ be an odd integer, and let $f\in S_{2k}^{\new}(N)$, $g\in S_{k+1}^{\new}(N)$ be two normalised newforms of level $N$ and weights $2k$ and $k+1$, respectively. Then
    \[
        \Lambda(f \otimes \mathrm{Ad}(g),k) = C(f,g) \cdot  \frac{\langle f,f \rangle}{\langle h,h \rangle} \cdot \frac{|\langle \varpi(F), g \times g\rangle|^2}{\langle g, g \rangle^2},
    \]
    where $C(f,g) = 2^{k+1} N^{-1} \prod_{q\mid N} (1+q)^2$, and:
    \begin{itemize}
        \item $h \in S_{k+1/2}(N)$ is any half-integral weight modular form in Shimura--Shintani correspondence with $f$;
        \item $F \in S^{(2)}_{k+1}(N)$ is the Saito--Kurokawa lift of $h$;
        \item $\varpi(F) \in S_{k+1}(N)\otimes S_{k+1}(N)$ is the pullback of $F$ to $\mathcal H\times \mathcal H$ embedded diagonally into Siegel's upper half space $\mathcal H_2$ (see Section \ref{sec:Pullback} for details).
    \end{itemize}
\end{theorem}

The Shimura--Shintani correspondence and the Saito--Kurokawa lift are recalled in Section \ref{sec:correspondences}. A corollary of the above central value formula is the algebraicity property of that value when divided by the corresponding period, as predicted by Deligne. Indeed, one has
\begin{equation} \label{eqn:PaldVP}
    \Lambda(f \otimes \mathrm{Ad}(g),k)^{\mathrm{alg}} := \frac{\Lambda(f \otimes \mathrm{Ad}(g),k)}{\langle g,g\rangle^2 \Omega_f^-} = C(f,g) \cdot  \frac{\langle f,f \rangle}{\langle h,h \rangle \Omega_f^{-}} \cdot \frac{|\langle \varpi(F), g \times g\rangle|^2}{\langle g, g \rangle^4} \in \Q(f,g),
\end{equation}
where $\Omega_f^- \in \C^{\times}$ denotes a complex period attached to $f$ by Shimura (see \cite[Theorem 4.8]{GS93} and equation \eqref{eqn:algebraicShintani} below) and $\Q(f,g)$ denotes the number field obtained by adjoining to $\Q$ the Fourier coefficients of $f$ and $g$. This explicit expression for the algebraic part of the central value $\Lambda(f\otimes \Ad(g),k)$ suggests the construction of the desired $p$-adic $L$-function $\mathcal L_p^{\circ}(\hf, \Ad(\hg))$ attached to two Hida families of ordinary $p$-stabilised newforms by directly interpolating the right hand side of \eqref{eqn:PaldVP} when $f$ and $g$ vary along the classical specialisations of $\hf$ and $\hg$. This is the strategy that we follow in this article. 

Before stating our main result, let us explain briefly what are the key steps in the aforementioned interpolation of (algebraic) central values, still with a classical flavour. Continue to assume that $f \in S_{2k}^{\new}(N)$ and $g \in S_{k+1}^{\new}(N)$ are normalised newforms of level $N$ and weights $2k$ and $k+1$, respectively, where $k\geq 1$ is odd and $N\geq 1$ is odd and squarefree. Let $p\nmid 2N$ be a prime and suppose that both $f$ and $g$ are {\em ordinary} at $p$, and write $\alpha_f := \alpha_p(f)$, $\beta_f:=\beta_p(f)$, and $\alpha_g := \alpha_p(g)$, $\beta_g := \beta_p(g)$ for the roots of the $p$-th Hecke polynomials of $f$ and $g$, respectively, labelled so that $\alpha_f$ and $\alpha_g$ are $p$-adic units. Then, consider the ordinary $p$-stabilisations
\[
f_{\alpha} \in S_{2k}(Np), \quad g_{\alpha} \in S_{k+1}(Np)
\]
of $f$ and $g$ on which $U_p$ acts as multiplication by $\alpha_f$ and $\alpha_g$, respectively. As we will see in Sections \ref{subsec:halfintegral} and \ref{sec:SiegelBasics}, we have notions of $p$-stabilisations $h_{\alpha}$ and $F_{\alpha}$ for the half-integral weight cuspform $h$ and the Siegel form $F$ as in Theorem \ref{thm:PaldVP Introduction}, respectively. 

In order to build the desired $p$-adic $L$-function we need to study the $p$-adic interpolation of the periods appearing on the right hand side of formula \eqref{eqn:PaldVP}. In Proposition \ref{prop:P}, we interpolate the periods
\[
    \frac{\langle f,f\rangle}{\langle h,h \rangle \Omega_f^-},
\]
by applying a classical formula due to Kohnen (recalled below in Theorem \ref{thm:KohnenFormula}) and using the existence of $\Lambda$-adic Shintani liftings as studied in \cite{CdVP21}. More challenging is the $p$-adic interpolation of the periods 
\[
 \frac{|\langle \varpi(F), g \times g\rangle|^2}{\langle g, g \rangle^4},
\]
which is performed in Proposition \ref{prop:JggF} and represents the main contribution of this article. On the classical side, the key ingredient is the proof of the following identity:
\begin{equation}\label{eqn:peterssonproducts-pullback-alpha}
    \frac{\langle e_\ord (\varpi(F_\alpha)), g_\alpha\times g_\alpha \rangle}{\langle g_\alpha, g_\alpha\rangle^2}
    = \frac{\mathcal E^{\circ}(f,\Ad(g))}{
    \mathcal E(\Ad(g))}
    \cdot \frac{\langle \varpi(F),g\times g \rangle}{\langle g,g\rangle^2},
\end{equation}
where $\eord$ denotes the ordinary projector acting on $S_{k+1}(Np) \otimes S_{k+1}(Np)$ (as described in Section \ref{subsec:HidaGL2xGL2}), and 
\begin{equation}\label{Efg}
        \mathcal E^{\circ}(f,\Ad(g)) := \left(1-\frac{\beta_f}{p^k}\right)\left(1-\frac{\beta_f\beta_g/\alpha_g}{p^k}\right), \quad 
        \mathcal E(\Ad(g)) =\left( 1-\frac{\beta_g}{\alpha_g}\right)\left(1-\frac{\beta_g}{p\alpha_g}\right).
\end{equation}
It is worth noticing that formula \eqref{eqn:peterssonproducts-pullback-alpha} follows from the computations of Section \ref{sec:Pullback} on pullbacks of Saito--Kurokawa lifts. More precisely, one needs an explicit relation between the pullbacks of $F$ and $U_pF$; this is the content of Theorem \ref{thm:PullbackUF}, which has a completely classical flavour and it is of independent interest. It would be nice to extend this kind of result for a general Siegel modular form. 

Having said this, let us move on towards stating the main result of this note. Let $p$ be an odd prime not dividing $N$, and consider the usual Iwasawa algebra $\Lambda:= \Zp[[\Gamma]]$, where $\Gamma = 1+p\Zp$, together with its associated weight space
\[
    \mathcal W := \Hom_{\mathrm{cts}}(\Lambda,\C_p).
\]
We define the subset $\mathcal W^{\cl} \subseteq \mathcal W$ of classical points in $\mathcal W$ as the image of the map 
\[
\Z_{\geq 1} \, \hookrightarrow \, \mathcal W, \quad k \mapsto \nu_{k-1}
\]
sending an integer $k \geq 1$ to the homomorphism $\nu_{k-1}: \Lambda \to \C_p$ determined by requiring that $\nu_k([t]) = t^{k-1}$ for all $t \in \Gamma$. This yields a dense subset of $\mathcal W$ for the Zariski topology. We will call $\mathrm{wt}(\nu_{k-1})=k-1$ the {\em weight} of the classical point $\nu_{k-1}$.

Let $\hf$ and $\hg$ be two Hida families of $p$-stabilised newforms of tame level $N$, and let us assume for simplicity along the introduction that their coefficients lie in $\Lambda$ (this last assumption is solely for sake of simplicity here, and will not be considered in the body of the text), so that we can see them as functions on $\mathcal W_{\hf} = \mathcal W$ and $\mathcal W_{\hg} = \mathcal W$. Consider the embedding
\begin{equation}\label{embeddingW:intro}
    \mathcal W \hookrightarrow \mathcal W_\hf \times \mathcal W_\hg
\end{equation}
defined as the only continuation of the one defined on classical weights by $\nu_{k-1} \mapsto (\nu_{2k-2},\nu_{k-1})$. By restricting to classical weights contained in a single residue class modulo $p-1$, we may assume that
\begin{itemize}
    \item $\hg(\nu_{k-1})$ is the $p$-stabilisation of an eigenform $g_{k+1} \in S_{k+1}^{\new}(N)$, for all $\nu_{k-1}\in \mathcal W^\cl$;
    \item $\hf(\nu_{k-1})$ is the $p$-stabilisation of an eigenform $f_{2k} \in S_{2k}^{\new}(N)$, for all $\nu_{k-1} \in \mathcal W^\cl$.
\end{itemize}
In this simplified setting, the main theorem of this article reads as follows.

\begin{theorem}\label{thm:mainthm-intro}
    Let $\hf$ and $\hg$ be two Hida families of ordinary $p$-stabilised newforms of tame level $N$ as above, with $N\geq 1$ odd and squarefree. There exists a unique $p$-adic analytic function $\Lp^{\circ}(\hf, \Ad(\hg))\in \mathrm{Frac}(\Lambda)$, defining a function on a Zariski-open subset $\mathcal U\subset \mathcal W$, such that for every classical point $\nu_{k-1}\in \mathcal W^{\cl}\, \cap \, \mathcal U$ the following interpolation formula holds:
    \[
        \Lp^{\circ}(\hf,\Ad(\hg))(\nu_{k-1}) =  \Omega_{2k}^- \cdot \mathscr C(N,k)^{-1} \cdot \frac{\mathcal E^{\circ}(f_{2k},\Ad(g_{k+1}))^2}{\mathcal E(\Ad(g_{k+1}))^2} \cdot 
        \Lambda(f_{2k}\otimes \Ad(g_{k+1}),k)^{\mathrm{alg}},
    \]
    where 
    \begin{itemize}
        \item $\Omega_{2k}^-$ is the $p$-adic period defined as in \cite{GS93} (see Theorem \ref{thm:mainGS} below);
        \item $\mathscr C(N,k) = (-1)^{[k/2]} 2^{2k} N^{-1} \prod_{q\mid N}(1+q)^2$;
        \item $\mathcal E^{\circ}(f_{2k},\Ad(g_{k+1}))$ and $\mathcal E(\Ad(g_{k+1}))$ are the Euler factors defined as in \eqref{Efg}. 
    \end{itemize}
\end{theorem}

We refer the reader to Section \ref{sec:padicLfunction}, and particularly to Theorem \ref{thm:Main Theorem}, for the precise statement in a general form. Let us just indicate briefly that, when dropping the assumption that $\hf$ and $\hg$ have coefficients in $\Lambda$, the embedding of $\mathcal W$ into $\mathcal W_{\hf} \times \mathcal W_{\hg}$ in \eqref{embeddingW:intro} is to be replaced with a natural embedding 
\[
    \mathcal W_{\hf} \times_{\mathcal W,\sigma} \mathcal W_{\hg} \, \hookrightarrow \, \mathcal W_{\hf} \times \mathcal W_{\hg},
\]
where $\sigma: \Lambda \, \to \, \Lambda$ is the isomorphism of $\Z_p$-algebras defined on group-like elements by $\sigma([t]) = [t^2]$.

\begin{remark}\label{rmk:twovariable-intro}
    The reason for the ornament $\circ$ in the notation $\mathcal L_p^{\circ}(\hf, \Ad(\hg))$ is that this is expected to be a one-variable restriction of a more general two-variable $p$-adic $L$-function $\mathcal L_p(\hf,\Ad(\hg)): \mathcal W_{\hf} \times \mathcal W_{\hg} \to \C_p$ satisfying an analogous interpolation property on the subset of $\mathcal W_{\hf}^{\cl} \times \mathcal W_{\hg}^{\cl}$ determined by $2\mathrm{wt}(\lambda)\geq \mathrm{wt}(\kappa)$. This two-variable $p$-adic $L$-function $\mathcal L_p(\hf, \Ad(\hg))$ can be constructed along similar lines as $\mathcal L_p^{\circ}(\hf,\Ad(\hg))$, by considering Hida families of (Shimura--Maass) derivatives of modular forms and interpolating nearly-holomorphic Saito--Kurokawa lifts. On the classical side, one can use the extension of Theorem \ref{thm:PaldVP Introduction} to the setting where $f$ and $g$ have weights $2k$ and $\ell+1$ with $\ell \geq k$ (see \cite{PaldVP2}). The fact that we interpolate on the line $k=\ell$ is the reason why $\mathcal E^{\circ}(f,\Ad(g))$ divides the expected Euler factor arising from the general conjectural description of the $p$-adic L-function of an arbitrary motive, but it is not equal to it.
\end{remark}

As a direct consequence of Theorem \ref{thm:mainthm-intro}, we prove in Section \ref{sec:factorisation} a factorisation of $p$-adic $L$-functions. Namely, we show that a suitable one-variable restriction of the so-called `balanced' triple product $p$-adic $L$-function $\mathcal L_p^{\mathrm{bal}}(\hf,\hg,\hg)$ associated with $(\hf,\hg,\hg)$ factors as a product of a suitable one-variable restriction of a Greenberg--Stevens $p$-adic $L$-function attached to $\hf$ and the $p$-adic $L$-function $\mathcal L_p^{\circ}(\hf,\Ad(\hg))$. This factorisation mirrors the factorisation of complex $L$-series
\[
L(f\otimes g \otimes g,s) = L(f, s-k)L(f\otimes \Ad(g),s-k)
\]
for $f \in S_{2k}^{\new}(N)$ and $g \in S_{k+1}^{\new}(N)$, arising by Artin formalism from the decomposition 
\[
V_g \otimes V_g \simeq \det(V_g) \otimes (\mathbf 1 \oplus \Ad(V_g)).
\]

Indeed, let us consider, under the appropriate sign assumption, the three-variable $p$-adic $L$-function  
\[
\mathcal L_p^{\mathrm{bal}}(\hf,\hg,\hg): \mathcal W_{\hf,\hg,\hg} := \mathcal W_{\hf} \times \mathcal W_{\hg} \times \mathcal W_{\hg} \, \longrightarrow \, \C_p
\]
described in \cite{Hsieh}, interpolating the algebraic parts of the central values of the Garret--Rankin complex $L$-series  $L(f_{\kappa} \otimes g_{\lambda} \otimes g_{\mu}, s)$, when $(\kappa,\lambda,\mu)$ varies in the balanced subregion of $\mathcal W_{\hf,\hg,\hg}^{\cl}$  (cf. Section \ref{sec:factorisation} for details).

With the above notation, let us denote by $\mathcal L_p^{\mathrm{bal},\circ}(\hf,\hg,\hg)$ the one-variable restriction of $\mathcal L_p^{\mathrm{bal}}(\hf,\hg,\hg)$ to $\mathcal W$ via the embedding 
\[
\mathcal W \, \hookrightarrow \, \mathcal W_{\hf,\hg,\hg} := \mathcal W_{\hf} \times \mathcal W_{\hg} \times \mathcal W_{\hg}
\]
naturally induced by the composition of \eqref{embeddingW:intro} and the diagonal embedding of $\mathcal W_\hg$. Besides, we choose in Section \ref{sec:factorisation} a one-variable $p$-adic $L$-function
\[
\mathcal L_p^{\circ}(\hf): \mathcal W \to \C_p
\]
constructed as a pullback of a Greenberg--Stevens $p$-adic $L$-function attached to $\hf$ to a line, interpolating the algebraic parts of the central values of $L(f_{2k},k)$, as $\kappa=\nu_{2k-2}$ varies in $\mathcal W^{\cl}$. 

\begin{theorem}\label{thm:factorisation-intro}
Assume that $\varepsilon(\hf) = +1$. Then we have the following factorisation of $p$-adic L-functions:
\[
    \mathcal L_p^{\mathrm{bal},\circ}(\hf,\hg,\hg)^2 = \mathcal L_p^{\circ}(\hf, \Ad(\hg)) \cdot \mathcal L_p^\circ(\hf) \pmod{\Lambda^\times}.
\]
\end{theorem}

The proof of this result follows straightforward by comparing the interpolation properties of the three $p$-adic $L$-functions involved, and the implicit $\Lambda$-adic function omitted in the formula is actually computable and it is described explicitly in Section \ref{sec:factorisation}. This formula suggests an alternative definition of our $p$-adic $L$-function as a quotient of $\mathcal L_p^{\mathrm{bal},\circ}(\hf,\hg,\hg)^2$ by $\mathcal L_p^{\circ}(\hf)$. This definition naturally extends to the two-variable case (and even to three variables, see Section \ref{sec:factorisation}), but only makes sense under the appropriate sign condition while our construction for $\mathcal L_p^{\circ}(\hf,\Ad(\hg))$ is unconditional. In fact, as commented in Remark \ref{rmk:twovariable-intro}, we expect a direct construction of a two-variable $p$-adic $L$-function $\mathcal L_p(\hf,\Ad(\hg))$ by extending the techniques in this note, and then a proof of a factorisation analogous to the one in Theorem \ref{thm:factorisation-intro} would proceed along the same lines.

More interesting and challenging is the case when $\varepsilon(\hf) = -1$. In this scenario, we expect a factorisation involving $\mathcal L_p^{\circ}(\hf,\Ad(\hg))$ (and more generally, $\mathcal L_p(\hf,\Ad(\hg))$) similar to the above one, where the balanced $p$-adic $L$-function is to be replaced by an {\em unbalanced} triple-product $p$-adic  $L$-function. The proof of the expected factorisation will not follow by comparing interpolation properties, as the interpolation regions will now be disjoint, but rather by following a strategy inspired by Dasgupta's and Gross' factorisation results in \cite{Das16}, \cite{Gro80}.

\vspace{0.2cm}

Finally, let us close this introduction by explaining the motivation behind some of the choices that are made for the construction of $\mathcal L_p^{\circ}(\hf,\Ad(\hg))$. One of the two main objects we use is a $\Lambda$-adic family  $\pmb\Theta \in \Lambda[[q]]$ of half-integral weight modular forms. A version of it was first defined in \cite{Stevens}, but we shall use here the version introduced in \cite{CdVP21}. The main reason for our choice is that while the first construction could be trivial, the latter can be guaranteed to be nonzero, which is fundamental for the construction of the $p$-adic $L$-function.

Another mild innovation of the present manuscript is the introduction of a completely explicit element $\hSK\in \Lambda[[q_1,q_2,\zeta]]$ interpolating the Saito--Kurokawa lifts, which is used to build the $\Lambda$-adic object interpolating the left hand side of equation \eqref{eqn:peterssonproducts-pullback-alpha}. This object is interpolating the elements $F_\alpha$, which should correspond to a suitable normalisation of the semi-ordinary $p$-stabilisation of the classical Saito--Kurokawa lift described in \cite{SU06}. An explicit $p$-adic interpolation of families of coefficients of Saito--Kurokawa lifts was also explored in \cite{Gue00} and in \cite{LN13}, but that approach was not useful for our purposes in this paper. Our construction is more reminiscent of the one in \cite{Kaw}, which is done for the Duke--Imamo\={g}lu lifting in level $1$.

\vspace{0.3cm}

\noindent {\bf Acknowledgements.} The first author acknowledges the support of the Government of Ireland Postdoctoral Fellowship GOIPD/2019/877, and the second author has received funding from the European Research Council (ERC) under the European Union’s Horizon 2020 research and innovation programme (grant agreement No 682152). We also thank Aprameyo Pal and Giovanni Rosso for useful conversations and their detailed feedback on an early version of this note.

\section{Modular forms}

\subsection{Classical modular forms (of integral weight)} \label{sec:ClassicalBasics}

Let $M, k \geq 1$ be integers, and let $S_{2k}(M)$ denote the space of classical modular cuspforms of weight $2k$ and trivial nebentype character for the congruence subgroup $\Gamma_0(M)$ of $\SL_2(\Z)$. If $\phi \in S_{2k}(M)$, we write its $q$-expansion as usual,
\[
    \phi = \sum_{n\geq 1} a_n(\phi) q^n, \qquad \text{where } q = \exp(2\pi i \tau).
\]
If $\phi, \psi \in S_{2k}(M)$, we normalize their Petersson product to be 
\[
    \langle \phi, \psi \rangle = \frac{1}{[\SL_2(\Z):\Gamma_0(M)]} \int_{\Gamma_0(M)\backslash \mathcal H} \phi(z)\overline{\psi(z)} y^{2k-2} dxdy.
\]
With this normalisation, the Petersson product of $\phi$ and $\psi$ does not depend on $M$, in the sense that it remains invariant if we replace $M$ by a multiple of it and see $\phi$ and $\psi$ as forms of that level.

Recall that if $\gamma \in \GL_2(\Q)$, then the slash operator (of weight $2k$) is defined by 
\[
    \phi|_{2k} \gamma (\tau) = \det(\gamma)^{2k-1}(c\tau + d)^{-2k}\phi(\gamma \tau), \quad \gamma = \begin{pmatrix} a & b \\ c & d\end{pmatrix},
\]
where $\det: \GL_2 \to \mathbf G_m$ (the determinant) is the similitude morphism or scale map.

Let $p$ be a prime, and consider the double coset operator $\Gamma_0(M)u_p\Gamma_0(M)$ associated with the element $u_p := \mathrm{diag}(1,p) \in \GL_2(\Q)$. As it is customary, when $p\nmid M$ this operator will be denoted by $T_p$, while if $p\mid M$ we will call it $U_p$. The action of $U_p$ can be described as
\[
    U_p\phi(\tau) = \sum_{b\in \Z/p\Z} \phi|_{2k} \alpha_b (\tau), \qquad \alpha_b = \begin{pmatrix} 1 &b \\ 0 &p \end{pmatrix} \in \GL_2(\Q),
\]
where $b$ runs over a set of representatives for $\Z/p\Z$. By making explicit the slash action of the elements $\alpha_b$ this can also be written as
\begin{equation} \label{eqn:UpexplicitGL2}
    U_p\phi(\tau) = p^{-1} \sum_{b\in \Z/p\Z} \phi( (\tau + b ) / p ).
\end{equation}
This last expression also makes sense for cuspforms $\phi \in S_{2k}(M)$ with $M$ not divisible by $p$; in that case, $U_p\phi$ belongs to $S_{2k}(Mp)$.

Continue to fix the prime $p$. Working on $q$-expansions, we can consider the operators $U$ and $V$ defined by
\[
    U\phi := \sum_n a_{pn}(\phi) q^n, \qquad V\phi := \sum_n a_n(\phi) q^{pn} = \sum_n a_{n/p}(\phi) q^n,
\]
where in the latter expression we read $a_{n/p}(\phi) = 0$ if $p\nmid n$. It is an easy exercise to see that $UV = 1$ while $VU$ does not, so that $V$ provides a right inverse to $U$. One can check that the action of $U$ is the same as that of the expression \eqref{eqn:UpexplicitGL2}, so that $U\phi\in S_{2k}(\mathrm{lcm}(M,p))$, and $U$ coincides with the Hecke operator $U_p$ when $p \mid M$. We also remark that $V=V_p := p^{1-2k}V_{p,2k-1}$, where $V_{p,2k-1}\phi = p^{2k-1} \phi(p\tau)$ is given by the (weight $2k$) slash action of the matrix $\mathrm{diag}(p,1)$, so that $V\phi\in S_{2k}(Mp)$. This is why we will often write simply $U$ and $V$ and no confusion should arise.

Suppose that $p$ does not divide $M$. Then one has the well-known relation
\begin{equation} \label{eqn:TUV}
    T_p = U + p^{2k-1}V
\end{equation}
among the operators $T_p$, $U$, and $V$. If $\phi = \sum_n a_n(\phi)q^n \in S_{2k}(M)$ is a normalised eigenform, and $\alpha = \alpha_p(\phi)$, $\beta = \beta_p(\phi)$ denote the roots of the $p$-th Hecke polynomial of $\phi$, so that $\alpha + \beta = a_p(\phi)$ and $\alpha\beta = p^{2k-1}$, we write 
\[
    \phi_{\alpha} := (1-\beta V) \phi, \qquad \phi_{\beta} := (1-\alpha V) \phi
\]
for its so-called {\em $p$-stabilisations}. These forms belong to $S_{2k}(Mp)$, and $U$ acts on $\phi_{\alpha}$ (resp. $\phi_{\beta}$) as multiplication by $\alpha$ (resp. $\beta$).

We will write $S_{2k}^{\new}(M) \subseteq S_{2k}(M)$ for the subspace of newforms of level $M$. It is the subspace orthogonal to the subspace of {\em old} forms in $S_{2k}(M)$ arising from lower levels, with respect to the Petersson product. When $\phi$ is a normalised eigenform for all Hecke operators, and $M\mid M'$, we shall write
\[
    S_{2k}(M')[\phi] := \{ \phi' \in S_{2k}(M') \mid T_p \phi' = a_p(\phi) \phi', \, \forall p\nmid M' \}.
\]
We will write $e_\phi$ for the projection onto the $\phi$-component $S_{2k}(M')[\phi]$. By the strong multiplicity one theorem, if $\phi \in S_{2k}^{\new}(M)$ is a newform then  $S_{2k}(M)[\phi] = \langle \phi \rangle$. From equation \eqref{eqn:TUV} and the definition of the $p$-stabilisations, one can also check that 
\begin{equation} \label{eqn:oldspace}
    S_{2k}(Mp)[\phi] = \langle \phi, V\phi \rangle = \langle U\phi,\phi \rangle = \langle \phi_\alpha, \phi_\beta \rangle.
\end{equation}

\subsection{Half-integral weight modular forms} \label{subsec:halfintegral}

Let $M\geq 1$ and $k\geq 1$ be integers, and assume for simplicity now that $M$ is {\em odd} and {\em squarefree}. We will write $\mathfrak S_{k+1/2}(M)$ for the space of cuspforms of half-integral weight $k+1/2$, level $\Gamma_0(4M)$, and trivial nebentype character, in the sense of Shimura \cite{Shi73} (note that we omit the $4$ in the notation as in \cite{Koh82}). If $h_1, h_2 \in \mathfrak S_{k+1/2}(M)$, their Petersson product is defined analogously to the case of integral weight modular forms, namely 
\[
    \langle h_1, h_2 \rangle = \frac{1}{[\SL_2(\Z):\Gamma_0(4M)]} \int_{\Gamma_0(4M)\backslash \mathcal H} h_1(z) \overline{h_2(z)} y^{k-3/2}dxdy.
\]
We will write $S_{k+1/2}(M) \subseteq \mathfrak S_{k+1/2}(M)$ for the so-called `Kohnen's plus subspace', which consists of those forms $h\in \mathfrak S_{k+1/2}(M)$ having $q$-expansion
\[
    h = \sum_{\substack{n\geq 1,\\(-1)^rn \equiv 0,1 \, (4)}} c(n) q^n.
\]
That is to say, the $n$-th Fourier coefficient is required to vanish if $(-1)^rn$ is not a discriminant\footnote{Very often, Kohnen's plus subspace is denoted $S_{k+1/2}^+(M)$, but we will drop the `$+$' from the notation as we will only work with this subspace.}. 

As described in \cite{Shi73} and \cite{Koh82}, there is a theory of Hecke operators similar to that of integral weight modular forms. For the purposes of this note, we just describe the analogues of the operators $T_p$, $U$, and $V$ described above for integral weight. If $p$ is an odd prime not dividing $M$, the Hecke operator $T_{p^2}$ acting on $S_{k+1/2}(M)$ (analogue of $T_p$) is described on $q$-expansions as
\[
    T_{p^2}\left(\sum_{n\geq 1} c(n)q^n\right) = \sum_{n\geq 1}\left(c(p^2n)  + p^{2k-1}c(n/p^2) + \left(\frac{(-1)^rn}{p}\right) p^{r-1} c(n)\right)q^n,
\]
where we read $c(n/p^2) = 0$ if $n/p^2$ is not an integer. If $p$ divides $M$, then one has an operator $U_{p^2}$ acting on $S_{k+1/2}(M)$ (analogue of $U$), which on $q$-expansions reads 
\[
    U_{p^2}\left(\sum_{n\geq 1} c(n)q^n\right) = \sum_{n\geq 1}c(p^2n) q^n.
\]
Finally, if $p$ is any prime, then the operator $V_{p^2}$ (analogue of $V$) defined on $q$-expansions by 
\[
    V_{p^2}\left(\sum_{n\geq 1} c(n)q^n\right) = \sum_{n\geq 1}c(n/p^2) q^n
\]
maps $S_{k+1/2}(M)$ into $S_{k+1/2}(Mp^2)$. This implies that $U_{p^2} = U^2$ and $V_{p^2} = V^2$, with the notation of the previous section. Let us write $\varepsilon_p(n) := \left(\frac{n}{p}\right)$ for the Legendre symbol. By the expression of the above operator one can rephrase the action of $T_{p^2}$ as
\begin{equation} \label{eqn:TUVtwist}
    T_{p^2}h = U_{p^2}h + p^{2k-1}V_{p^2}h + (-1)^{r(p-1)/2} h\otimes \varepsilon_p,
\end{equation}
where $h \otimes \varepsilon_p$ has $q$-expansion
\[
\sum_{n\geq 1} \varepsilon_p(n)c(n)q^n,
\]
if $c(n)$ denotes the $n$-th Fourier coefficient of $h$. Note that the identity \eqref{eqn:TUVtwist} shows that the twisted form $h\otimes \varepsilon_p$ belongs to $S_{k+1/2}(Mp^2)$ (see \cite[Lemma 3.6]{Shi73} for a more general result about twists).

Because of our assumption that $M$ is squarefree, on $S_{k+1/2}(M)$ there is a well-behaved theory of oldforms and newforms (see \cite{Koh82}), and we will write $S_{k+1/2}^{\new}(M) \subseteq S_{k+1/2}(M)$ for the subspace of newforms. In particular, if $h \in S_{k+1/2}^{\new}(M)$ is an eigenform for all Hecke operators and $p\nmid M$, then the $h$-isotypical subspace $S_{k+1/2}(Mp)[h] \subseteq S_{k+1/2}(Mp)$ is the two-dimensional subspace $\langle h, U_{p^2}h\rangle$ spanned by $h$ and $U_{p^2}h$.

Using equation \eqref{eqn:TUVtwist} and the fact that $U_{p^2}(h\otimes \varepsilon_p)=0$, one can diagonalize the action of $U_{p^2}$. Indeed, let $a_p(h)$ be the eigenvalue of $h$ for $T_{p^2}$ and let $\alpha = \alpha_p(h)$, $\beta =\beta_p(h)$ be such that $\alpha+\beta = a_p(h)$ and $\alpha\beta = p^{2k-1}$. Then one has $S_{k+1/2}(Mp)[h] = \langle h_\alpha, h_\beta\rangle$, where:
\[
    h_\alpha := \alpha^{-1}(U_{p^2} - \beta) h, \qquad h_\beta := \beta^{-1}(U_{p^2} - \alpha) h.
\]

\begin{proposition} \label{prop:halpha}
    We have the following equality:
    \[
    h_\alpha 
    = 
    h - (-1)^{(p-1)/2}p^{-k}\beta h\otimes \varepsilon_p- \beta V^2h.
\]
\end{proposition}

\begin{proof}
    This easily follows from equation \eqref{eqn:TUVtwist}.
\end{proof}

\begin{remark}
    We want to remark that in the half-integral case one needs to use the operator $U_{p^2}$ to produce old forms, and not the operators $U_p$ or $V_p$, which have different target (cf. \cite[Prop. 1.3, 1.5]{Shi73}). Indeed, one has
    \[
        U_p, V_p : S_{k+1/2}(M) \to S_{k+1/2}(Mp, \chi_p),
    \]
    where $\chi_p:= \left( \frac{p}{-}\right)$. In particular, one cannot define the $p$-stabilisation using $V_p$ or $V_{p^2}$; our choice has been normalised in such a way that it is compatible with that of classical modular forms.
\end{remark}

\subsection{Siegel modular forms} \label{sec:SiegelBasics}

Let $M, k \geq 1$ be {\em odd} integers. We write $S_{k+1}^{(2)}(M)$ for the space of (genus two) Siegel forms of weight $k+1$ and level $\Gamma_0^{(2)}(M)$, where 
\[
\Gamma_0^{(2)}(M) := \left\lbrace \begin{pmatrix} A & B \\ C & D \end{pmatrix} \in \mathrm{Sp}_4(\Z): C \equiv 0_2 \pmod p \right\rbrace \subseteq \mathrm{Sp}_4(\Z)
\]
is the Hecke-type congruence subgroup of level $M$. Writing elements of Siegel's upper half-space $\mathcal H_2$ of genus two as symmetric matrices 
\[
Z = \left(\begin{array}{cc} \tau_1 & z \\ z & \tau_2 \end{array}\right), \quad \tau_i \in \mathcal H, \, z \in \C, \, \mathrm{Im}(\tau_1\tau_2 - z^2) > 0,
\]
we may regard Siegel forms $\Phi \in S_{k+1}^{(2)}(M)$ as functions of $(\tau_1,z,\tau_2)$. We will often write $Z = (\tau_1,z,\tau_2)$ with the obvious meaning to simplify notation. 

If $\Phi \in S_{k+1}^{(2)}(M)$, it is well-known that $\Phi$ admits a $q$-expansion that reads 
\[
    \Phi = \sum_{B>0} A(B) q^B, \quad \text{where } q^B = \exp(2\pi i \mathrm{Tr}(BZ)),
\]
where $B$ runs over all positive definite, half-integral symmetric matrices 
\[
    B = \begin{pmatrix} m &r/2 \\ r/2 &n \end{pmatrix}.
\]

If $\Phi \in S_{k+1}^{(2)}(M)$ and $\gamma \in \operatorname{GSp}_4(\Q)$, then the slash operator (of weight $k+1$) is defined by 
\[
    \Phi|_{k+1} \gamma (Z) = \mu(\gamma)^{2k-1}\det(CZ + D)^{-(k+1)}\Phi(\gamma Z), \quad \gamma = \left(\begin{array}{cc} A & B \\ C & D\end{array}\right)
\]
where $\mu: \operatorname{GSp}_4 \to \mathbf G_m$ is the similitude morphism or scale map. 

For a prime $p$, consider the double coset operator $\Gamma_0^{(2)}(M)u_p^{(2)}\Gamma_0^{(2)}(M)$ associated with the element $u_p^{(2)}= \mathrm{diag}(1,1,p,p) \in \operatorname{GSp}_4(\Q)$. When $p \nmid M$, we denote this operator by $T_p$, while if $p \mid M$ we call it $U_p$. In the literature, the latter is sometimes referred to as\footnote{In contrast to $U_{p,1}$, which is defined analogously with $\mathrm{diag}(1,p,p^2,p)$.} $U_{p,0}$. Similarly as in the case of classical modular forms, the action of $U_p$ can be described as an average of slash operations. Namely, if $p$ divides $M$ and $\Phi \in S_{k+1}^{(2)}(M)$ then
\[
    U_p\Phi (Z) = \sum_{B} \Phi|_{k+1}\alpha_B(Z), \quad \alpha_B = \begin{pmatrix} \mathrm{Id}_2 & B \\ 0 & p\mathrm{Id}_2 \end{pmatrix} \in \GSp_4(\Q),
\]
where $B \in \mathrm{Sym}_2(\Z)$ runs over a set of representatives for the set $\mathrm{Sym}_2(\Z/p\Z)$ of symmetric two-by-two matrices with coefficients in $\Z/p\Z$. By making explicit the slash action of the elements $\alpha_B$, this can be rewritten as 
\begin{equation} \label{eqn:Upexplicit}
    U_p\Phi(Z) 
    = p^{-3} \sum_B \Phi(\alpha_B Z) 
    = p^{-3} \sum_B \Phi\left(\frac{Z+B}{p}\right).
\end{equation}
This last expression can also be applied to forms in $S_{k+1}^{(2)}(M)$ with $M$ not divisible by $p$; in that case, the resulting form belongs to $S_{k+1}^{(2)}(Mp)$.

Finally, if $p$ is an arbitrary prime and $\Phi \in S_{k+1}^{(2)}(M)$, the operator $V_p$ can be defined by setting $V_p\Phi(Z) = \Phi(pZ)$. Equivalently, $V_p = p^{1-2k}V_{p,k+1}$, where 
\[
V_{p,k+1}\Phi(Z) = \Phi|_{k+1}\gamma(Z) = p^{2k-1}\Phi(pZ), \quad \gamma = \mathrm{diag}(p,p,1,1).
\]

As for the $\GL_2$ case, on the space of genus two Siegel modular forms we can define the operators
\[
    UF = \sum_B A(pB) q^B, \qquad VF = \sum_B A(B) q^{pB} = \sum_B A(B/p) q^B
\]
by their action on $q$-expansions, where in the latter case we read $A(B/p)= 0$ whenever $p\nmid B$. It is a straightforward computation to check that $U$ acts as in equation \eqref{eqn:Upexplicit}, so it coincides with the Hecke operator $U_p$ when $p$ divides $M$, and that $V$ acts as the operator $V_p$. For this reason, and similarly as we do for classical forms, if the prime $p$ is clear from the context and there is no risk of confusion we will just write $U$ and $V$ for $U_p$ and $V_p$, respectively.

\section{Shimura--Shintani correspondence and Saito--Kurokawa lifts}\label{sec:correspondences}

Changing slightly the notation from the previous paragraphs, in this section we fix integers $N, k \geq 1$, and assume for simplicity that $N$ is odd and squarefree.

\subsection{Shimura--Shintani correspondence and \texorpdfstring{$\d$}{d}-th Shintani liftings} \label{sec:ShintaniBasics}

One of the main reasons for the arithmetic interest of half-integral weight modular forms stems from the so-called Shimura--Shintani correspondence. 

\begin{theorem}[Shimura--Shintani correspondence]
Let $N, k \geq 1$ be as above. There is a Hecke-equivariant linear isomorphism
\begin{equation}\label{isom:ShiShi}
S_{2k}^{\new}(N) \, \stackrel{\simeq}{\longrightarrow} \, S_{k+1/2}^{\new}(N).
\end{equation}
\end{theorem}

In particular, given a normalised newform $f \in S_{2k}^{\new}(N)$, on which the Hecke operators $T_{\ell}$ at primes $\ell \nmid N$ act with eigenvalue $a_{\ell}$, there is a unique $h \in S_{k+1/2}^{\new}(N)$, up to scalar, on which the Hecke operators $T_{\ell^2}$, for primes $\ell \nmid N$, act with eigenvalue $a_{\ell}$. If $h$ is such a half-integral weight modular form, one says that $h$ is in Shimura--Shintani correspondence with $f$, or that $h$ and $f$ are in Shimura--Shintani correspondence. An example of how half-integral weight modular forms carry over arithmetic information about integral weight modular forms via this correspondence is the following important formula due to Kohnen \cite[Corollary 1]{Koh85}, generalizing a previous formula of Kohnen--Zagier \cite[Theorem 1]{KZ81} for trivial level.

\begin{theorem}[Kohnen's formula]\label{thm:KohnenFormula}
Let $N$ and $k$ be as above. Let $h \in S_{k+1/2}^{\new}(N)$ be any non-zero cusp form in Shimura--Shintani correspondence with $f \in S_{2k}^{\new}(N)$, and let $D$ be a fundamental discriminant such that $(-1)^kD > 0$ and $\left(\frac{D}{\ell}\right) = w_{\ell}(f)$ for all primes $\ell \mid N$, where $w_{\ell}(f) = \pm 1$ denotes the eigenvalue of the $\ell$-th Atkin--Lehner involution acting on $f$. Then one has
\begin{equation}\label{KohnenFormula}
|c_{|D|}(h)|^2 = 2^{\nu(N)} \frac{(k-1)!}{\pi^k}|D|^{k-1/2} \cdot \frac{\langle h,h\rangle}{\langle f,f\rangle} \cdot L(f,D,k),
\end{equation}
where $\nu(N)$ is the number of primes dividing $N$.
\end{theorem}

The isomorphism in \eqref{isom:ShiShi} can be made explicit by means of the so-called {\em $\d$-th Shintani liftings}. Indeed, associated with each fundamental discriminant $\d$ such that $(-1)^k\d > 0$, the $\d$-th Shintani lifting is a Hecke-equivariant linear map 
\[
\theta_{k,N,\d}: S_{2k}(N) \, \longrightarrow \, S_{k+1/2}(N)
\]
from $S_{2k}(N)$ to Kohnen's plus subspace $S_{k+1/2}(N) \subseteq \mathfrak S_{k+1/2}(N)$. It is defined by means of certain geodesic cycle integrals on the complex upper-half plane, firstly studied by Shintani \cite{Shintani}, and it is adjoint to the so-called $\d$-th Shimura lifting with respect to the Petersson product. An appropriate combination of Shintani liftings, for various discriminants, provides a realisation of the isomorphism $S_{2k}^{\new}(N) \simeq S_{k+1/2}^{\new}(N)$ as in \eqref{isom:ShiShi} (see \cite[Theorem 2]{Koh82}). Also, the proof of Theorem \ref{thm:KohnenFormula} shows that the $|\d|$-th Fourier coefficient of the $\d$-th Shintani lifting is in fact a multiple of the twisted $L$-value $L(f,\d,k)$. Namely, one has (see \cite[p. 243]{Koh85}, \cite[Eq. (15)]{CdVP21})
\begin{equation}\label{KohnenFormula-finer}
c_{|\d|}(\theta_{k,N,\d}(f)) = (-1)^{[k/2]} 2^{\nu(N)+k}|\d|^k (k-1)!\cdot \frac{L(f,\d,k)}{(2\pi i)^k \mathfrak g(\chi_{\d})},
\end{equation}
where $[x]$ denotes the smallest integer such that $[x]\geq x$, and $\mathfrak g(\chi_{\d})$ is the Gauss sum of the quadratic character associated with $\d$. 

Another important feature of the $\d$-th Shintani lifting, especially for the purposes of $p$-adic interpolation and hence for the topic of this note, is the fact that it can be made algebraic. Namely, suppose that $f \in S_{2k}^{\new}(N)$, and choose a fundamental discriminant $\d$ with $(-1)^k\d > 0$ and $\theta_{k,N,\d}(f) \neq 0$ (such a choice is indeed possible: one can use the above formula \eqref{KohnenFormula-finer} together with well-known results on the non-vanishing of (twisted) central values of $L$-series of modular forms, e.g. from \cite{BFH}). A cohomological description of the Shintani lifting in terms of modular symbols, due originally to Stevens \cite{Stevens}, shows that if $f\in S_{2k}(N)$ is a Hecke eigenform then (cf. \cite[Section 5.1]{CdVP21})
\begin{equation}\label{eqn:algebraicShintani}
\theta_{k,N,\d}^{\mathrm{alg}}(f) := \frac{1}{\Omega_f^-}\theta_{k,N,\d}(f) \in S_{k+1/2}(N; \mathcal O_f),
\end{equation}
where $\mathcal O_f$ denotes the ring of integers in the Hecke field of $f$, and $\Omega_f^- \in \C^{\times}$ is a complex period attached to $f$ by Shimura. One can actually attach two complex periods $\Omega_f^{\pm} \in \C^{\times}$ satisfying the algebraicity property spelled out in \cite[Theorem 4.8]{GS93}, and with the additional property that $\Omega_f^+ \Omega_f^- = \langle f,f\rangle$. We do fix such a choice attached to each Hecke eigenform.

We end this paragraph by pointing out the following result.

\begin{proposition}\label{prop:shintani-alpha}
    Let $f\in S_{2k}^{\new}(N)$, $p\nmid 2N$ be a prime, and let $\alpha$ be a root of the Hecke polynomial of $f$ at $p$. If $p\mid \d$, then we have the following equality:
    \[
        \theta_{k,Np,\d}(f_{\alpha}) = \theta_{k,N,\d}(f)_{\alpha}.
    \]
\end{proposition}
\begin{proof}
    This essentially follows from the fact that $\theta_{k,Np,\d}(f) = \theta_{k,N,\d}(f)$, which is proven in \cite{CdVP21}.
\end{proof}

\subsection{Saito--Kurokawa lift} \label{sec:SKlifts}

The theory of Saito--Kurokawa liftings for arbitrary level and character in classical terms was established in \cite{Ibu12}, generalizing \cite{EZ85} and \cite{Maass1, Maass2, Maass3}. For the purpose of this note, we restrict ouselves to the case of odd squarefree level and trivial character. Thus we continue to assume as above that $N\geq 1$ is odd and squarefree, and suppose that $k \geq 1$ is odd as well. Fix also an odd prime $p$ not dividing $N$.

Let $f \in S_{2k}^{\new}(N)$ be a normalised eigenform, and $h \in S_{k+1/2}^{\new}(N)$ be any non-zero form in Shimura--Shintani correspondence with $f$ (for example, we can normalize $h$ to be the $\d$-th Shintani of $f$ for a particular choice of $\d$ as above, but it will not be necessary to fix such a choice now). The {\em Saito--Kurokawa lift} attaches to $h$ a Siegel modular form 
\[
    F := \SK_N(h) \in S_{k+1}^{(2)}(N)
\]
of weight $k+1$ and level $N$. This is sometimes referred to as the Saito--Kurokawa lift of $f$ as well (although it is only defined up to a scalar multiple). On $q$-expansions, it is defined by setting (see, e.g., \cite{Ibu12})
\begin{equation} \label{eqn:SKCoefficients}
    F = \sum_{B>0} A(B) q^B, \qquad \text{with } \, A(B) = \sum_{\substack{0<d\mid \gcd(B),\\(d,N)=1}} d^k c(\det(2B)/d^2),
\end{equation}
where $c(n)$ denotes the $n$-th Fourier coefficient in the $q$-expansion of $h$. Here, $\gcd(B)$ denotes the greatest common divisor of $m$, $r$, $n$ if $B$ has entries $m$, $r/2$, $n$ with the usual notation. Observe in particular that $A(B) = c(\det(2B))$ when $\gcd(B) = 1$.

From its very definition it is clear that $F = \SK_N(h)$ depends heavily on $N$, meaning that if we consider $h$ as an old form of level $Nt$ for some integer $t > 1$, then the Siegel forms $\SK_{Nt}(h)$ and $\SK_N(h)$ are a priori different in $S_{k+1}^{(2)}(Nt)$. 
Let us consider the Saito--Kurokawa lift $F^{(p)} := \SK_{Np}(h) \in S_{k+1}^{(2)}(Np)$ of $h$ in level $Np$. By definition, its $q$-expansion is 
\[
    F^{(p)} = \sum_{B>0} A^{(p)}(B) q^B, \qquad A^{(p)}(B) = \sum_{\substack{0<d\mid \gcd(B),\\(d,Np)=1}} d^k c(\det(2B)/d^2).
\]
An easy computation shows that $A^{(p)}(B) = A(B)-p^kA(B/p)$, for all $B$, which implies the following equality:
\begin{equation*}
    F^{(p)} = (1-p^kV)F.
\end{equation*}
This shows, in particular, that $\SK_N(h) \ne \SK_{Np}(h)$ (compare this with the identity $\theta_{k,Np,\d}(f) = \theta_{k,N,\d}(f)$ mentioned in the proof of Proposition \ref{prop:shintani-alpha}).

Similarly, let $\alpha_f$, $\beta_f$ be the roots of the $p$-th Hecke polynomial of $f$. We define
\[
    F_{\alpha} := \SK_{Np}(h_{\alpha}) \in S_{k+1}^{(2)}(Np)
\]
to be the Saito--Kurokawa lift of the $p$-stabilisation $h_{\alpha} \in S_{k +1/2}(Np)$ of $h$ corresponding to $\alpha_f$. This is a genus two Siegel form of weight $k+1$ and level $Np$, whose $q$-expansion is given by:
\[
    F_\alpha = \sum_{B>0} A_\alpha(B) q^B, \qquad A_\alpha(B) = \sum_{\substack{0<d\mid \gcd(B),\\(d,Np)=1}} d^k c_\alpha(\det(2B)/d^2),
\]
where $c_{\alpha}(n)$ denotes the $n$-th Fourier coefficient of $h_{\alpha}$. Note that the expression for $A_{\alpha}(B)$ reads exactly as the one for $A^{(p)}$, but with the $c(n)$'s replaced now by $c_{\alpha}(n)$'s. By definition, we know that $h_\alpha = \alpha_f^{-1}(U_{p^2}-\beta_f)$ and since the Saito--Kurokawa lift is Hecke equivariant (in particular, $\SK_{Np} \circ U_{p^2} = U_p \circ \SK_{Np}$), one deduces that
\begin{equation} \label{eqn:semiordinary}
    F_\alpha = \alpha_f^{-1}(U-\beta)F^{(p)}= \alpha_f^{-1}(U-\beta_f)(1-p^kV)F.
\end{equation}

\begin{remark} \label{rmk:semiordinary}
    One can similarly define the Siegel modular form
    \[
        F^{\mathrm{so}} := (U-\beta_f)(U-p^k) F \in S_{k+1}^{(2)}(Np).
    \]
    When $\alpha_f$ is a $p$-adic unit, the latter is sometimes referred to as the {\em semi-ordinary $p$-stabilisation} of $F$. One can easily check that $F_\alpha = \alpha_f^{-2}F^{\mathrm{so}}$, since $U$ acts invertibly on  $S^{(2)}_{k+1}(Np)[f]$. We refer the reader to \cite{Kaw} for more details on the Siegel form $F^{\mathrm{so}}$.
\end{remark}

\begin{proposition} \label{prop:Falpha} With the above notation, we have
    \[
        F_\alpha 
        =
        (1-\beta_f V)(1-p^kV)F - (-1)^{(p-1)/2}p^{-k} \beta_f F \otimes \varepsilon_p,
    \]
    where $F\otimes \varepsilon_p$ is the Siegel form defined by the $q$-expansion 
    \[
    F\otimes \varepsilon_p = \sum_{B>0} \varepsilon_p(\det(2B))A(B)q^B.
    \]
\end{proposition}
\begin{proof}
    The equality follows from Proposition \ref{prop:halpha}. Indeed, we have $h_{\alpha} = h - \sigma p^{-k}\beta_f h \otimes \varepsilon_p -\beta_fV^2h$, where $\sigma = (-1)^{(p-1)/2}$. Applying $\SK_{Np}$ we get
    \[
    F_{\alpha} = F^{(p)} - \beta_f \SK_{Np}(\sigma p^{-k}h \otimes \varepsilon_p + V^2h).
    \]
    If $c(n)$ denotes the $n$-th Fourier coefficient of $h$, then the $n$-th Fourier coefficient of $\sigma p^{-k}h \otimes \varepsilon_p + V^2h$ is $\sigma p^{-k}(\frac{n}{p})c(n) + c(n/p^2)$, and it easily follows that 
    \[
    \SK_{Np}(\sigma p^{-k}h \otimes \varepsilon_p + V^2h) = \sigma p^{-k}F^{(p)}\otimes\varepsilon_p + VF^{(p)}.
    \]
    Noting that $F^{(p)}\otimes\varepsilon_p = F \otimes \varepsilon_p$, the equality of the $q$-expansions follows. Since all the other terms in the identity are modular forms of level $Np^2$, we see that $F\otimes \varepsilon_p \in S_{k+1}^{(2)}(Np^2)$ as well. 
\end{proof}

\begin{corollary}\label{cor:UFrelation}
    We have the following equality:
    \[
        UF - (p^k+a_p(f)) F + (p^{2k-1}+p^ka_p(f)) VF - p^{3k-1}V^2F + (-1)^{(p-1)/2}p^{k-1} F\otimes \varepsilon_p = 0.
    \]
\end{corollary}
\begin{proof}
This follows straightforward from the above proposition, replacing $F^{(p)}$ by $(1-p^kV)F$ and using the relations $\alpha_f\beta_f = p^{2k-1}$, $\alpha_f + \beta_f = a_p(f)$. This result can also be directly derived by looking at the Fourier coefficients of $F$, using the properties of the Fourier coefficients of $h$.
\end{proof}

\section{On the pullback of Saito--Kurokawa lifts to the diagonal} \label{sec:Pullback}

This section is devoted to describe the pullback of Siegel forms to the diagonal $\mathcal H \times \mathcal H \subset \mathcal H_2$, where $\mathcal H$ denotes Poincaré's upper half-plane, and $\mathcal H_2$ stands for Siegel's upper half-space of genus two. In particular, we describe pullbacks of Saito--Kurokawa lifts as considered in the previous section. 

Before focusing on the specific case of Saito--Kurokawa lifts, we may first consider some generalitites about the pullback of Siegel forms. To do so, let us fix integers $k, M \geq 1$, with $k$ odd, and consider the spaces $S_{k+1}(M)$ and $S_{k+1}^{(2)}(M)$ of classical and Siegel cusp forms of weight $k+1$ and level $M$, respectively. Recall that $\mathcal H \times \mathcal H$ can be embedded {\em diagonally} in the Siegel upper half-space $\mathcal H_{2}$ of genus two via the map
\[
    (\tau_1, \tau_2) \mapsto \begin{pmatrix} \tau_1 &0 \\ 0 &\tau_2 \end{pmatrix}.
\]
This induces a linear map, referred to as {\em pullback to the diagonal},
\begin{align*}
    \varpi: S_{k+1}^{(2)}(M) \, &\rightarrow \, S_{k+1}(M) \otimes S_{k+1}(M) \\
    \Phi &\mapsto \varpi(\Phi)  = \Phi|_{\mathcal H \times \mathcal H},
\end{align*}
provided by the restriction to $z = 0$. In terms of $q$-expansions, and adopting the notation of the previous section, suppose that
\[
\Phi = \sum_{B} A(B) q^B,
\]
where $B$ runs over the half-integral positive definite symmetric two-by-two matrices, and write 
\[
A(n,r,m) = A\left(\begin{pmatrix} n & r/2 \\ r/2 & m\end{pmatrix}\right), \quad n,r,m \in \Z, 4nm-r^2 > 0.
\]
Then the pullback map is described on $q$-expansions by
\begin{equation} \label{eqn:q-expansionPullback}
    \varpi(\Phi) (\tau_1,\tau_2) 
    =
    \sum_{n,m\geq 1}\left( \sum_{\substack{r\in \Z,\\r^2 < 4nm}} A(n,r,m) \right)q_1^n q_2^m, \qquad q_j = \mathrm{exp}(2\pi i \tau_j).
\end{equation}
We may write $\alpha_{n,m}(\varpi(\Phi))$ for the $(n,m)$-th Fourier coefficient in this expansion. Observe that given a basis $\{ \phi_i \}_i$ of $S_{k+1}(M)$ we can write the pullback of $\Phi$ as
\[
    \varpi(\Phi) = \sum_{i,j} \lambda_{i,j} \phi_i(\tau_1) \times \phi_j(\tau_2), 
\]
where the scalars $\lambda_{i,j}$ are uniquely determined by $\Phi$.

\begin{lemma} \label{lemma:Vcommutes}
    If $p$ is a prime, then 
    \[
        \varpi(V\Phi) = V\times V \varpi(\Phi).
    \]
\end{lemma}
\begin{proof}
    It is clear from the definitions that $\varpi(\Phi(pZ)) = \varpi(\Phi)(p\tau_1,p\tau_2)$, hence the statement follows.
\end{proof}

\begin{remark}
More generally, consider the embedding $\iota: \mathrm G(\SL_2 \times \SL_2) \, \hookrightarrow \, \operatorname{GSp}_4$ given by
\[
    \left( 
    \begin{pmatrix} a_1 & b_1 \\ c_1 & d_1 \end{pmatrix},
    \begin{pmatrix} a_2 & b_2 \\ c_2 & d_2 \end{pmatrix} 
    \right) 
    \, \longmapsto \,
    \begin{pmatrix} 
    a_1 & 0 & b_1 & 0 \\ 0 & a_2 & 0 & b_2 \\ c_1 & 0 & d_1 & 0 \\ 0 & c_2 & 0 & d_2
    \end{pmatrix},
\]
where $\mathrm{G}(\SL_2\times\SL_2) = \{(g,h) \in \GL_2 \times \GL_2: \det(g)=\det(h)\}$. Then the product slash action on $S_{k+1}(M) \otimes S_{k+1}(M)$ with elements in $\mathrm G(\SL_2 \times \SL_2)$ can be compared with the slash action on $S_{k+1}^{(2)}(M)$ of their image under $\iota$. For example, consider the element
    \[
     (v_p,v_p) = \left(\left(\begin{array}{cc} p & 0 \\ 0 & 1 \end{array}\right), \left(\begin{array}{cc} p & 0 \\ 0 & 1 \end{array}\right)\right) \in \operatorname G(\SL_2(\Q)\times\SL_2(\Q)).
    \]
The slash action of this element on $S_{k+1}(M)\otimes S_{k+1}(M)$ is the operator $V_{p,k+1} \times V_{p,k+1} = p^{2k} V \times V$, while its image under the above embedding is $\gamma_p = \mathrm{diag}(p,p,1,1)$, and corresponds with the operator $V_{p,k+1} = p^{2k-1}V$ for $\GSp(4)$. Taking care of the powers of $p$ arising in each case, the above lemma can be rewritten as a comparison between the slash actions with $\gamma_p$ and $(v_p,v_p)$.
\end{remark}

The analogue of Lemma \ref{lemma:Vcommutes} for the operators $U$ and $U\times U$ does not hold, and this will be important from now on. To proceed in understanding how the $U$-operator behaves with respect to pullback, we restrict ourselves to the case of Saito--Kurokawa lifts. We thus assume from now on that $N\geq 1$ is odd and squarefree, $k\geq 1$ is an odd integer, and $p$ is an odd prime not dividing $N$. With this, let $F = \SK(h) \in S_{k+1}^{(2)}(N)$ be a Saito--Kurokawa lift as in the previous section. In this setting, recall that we have a well-defined twisted Siegel form $F \otimes \varepsilon_p \in S_{k+1}(Np^2)$ as in Proposition \ref{prop:Falpha}. The main result of the section is the following:

\begin{theorem}\label{thm:PullbackUF}
    Let $F = \SK_N(h) \in S_{k+1}^{(2)}(N)$ be a Saito--Kurokawa lifting and let $\phi \in S_{k+1}^{\new}(N)$ be a normalised eigenform. Let $\varpi_\phi := e_\phi\otimes e_\phi \circ \varpi$, where $e_\phi$ is the projector onto the $\phi$-component as defined in Section \ref{sec:ClassicalBasics}, and let $\lambda_\phi$ be such that $\varpi_\phi (F)= \lambda_{\phi} \phi\times \phi$. Then 
    \[
        \varpi_\phi(UF) = \lambda_{\phi}\left(A\phi \times \phi + B\phi\times V\phi + CV\phi\times\phi + D V\phi\times V\phi\right),
    \]
    where 
    \[
        A = \frac{(p^{k-1}(p-1) + a_p) a_p(\phi)^2 - p^k(p+1)(p^{k-1}(p+1)+a_p)}{a_p(\phi)^2 - p^{k-1}(p+1)^2},
    \]
    and $B$, $C$, $D$ are given by the formulae
    \[
        B = C = \frac{p a_p(\phi)}{p+1}(p^k-p^{k-1}+a_p-A), \quad D = -p^{k+1}(p^k+a_p-A).
    \]
\end{theorem}

As a consequence of the above theorem, we have the following:

\begin{corollary}\label{cor:Falpha}
Let $\phi \in S_{k+1}^{\new}(N)$ be a normalised eigenform, and suppose that $\varpi_{\phi}(F) = \lambda_{\phi} \phi\times \phi$. If $F_\alpha =\SK_{Np}(h_\alpha) \in S_{k+1}^{(2)}(Np)$ is a $p$-stabilisation of $F$ as described in Section \ref{sec:SiegelBasics}, then we have 
\[
    \varpi_{\phi}(F_\alpha) = \lambda_\phi \left(1-\frac{\beta}{p^k}\right)\left( A_{\alpha} \phi\times\phi + B_{\alpha}\phi\times V\phi + C_{\alpha}V\phi\times \phi + D_{\alpha}V\phi\times V\phi\right),
\]
where $\beta$ is the other eigenvalue of the $U_{p^2}$-action on $S_{k+1/2}(Np)[h]$,
\[
    A_{\alpha} =  1 - \frac{(p+1)\left(1-\frac{\beta}{p^{k-1}}\right)}{\left(p-\frac{\alpha_{\phi}}{\beta_{\phi}}\right)\left(p-\frac{\beta_{\phi}}{\alpha_{\phi}}\right)}
\]
and $B_{\alpha}$, $C_{\alpha}$, $D_{\alpha}$ are given by the formulae
\[
B_{\alpha} = C_{\alpha} = \frac{pa_p(\phi)}{p+1}(1-A_{\alpha}), \quad D_{\alpha} = p^{k+1}(A_{\alpha}-1)-p\beta.
\]
\end{corollary}
\begin{proof}
From equation \eqref{eqn:semiordinary}, we know that 
\[
\alpha F_{\alpha} = (U-\beta)(1-p^kV)F = UF - (p^k+\beta)F + \beta p^k VF.
\]
Therefore, 
\[
\alpha \varpi_\phi(F_{\alpha}) = \varpi_\phi(UF) - (p^k+\beta)\varpi_\phi(F) + \beta p^k V \times V \varpi_\phi(F)
\]
Letting $A$, $B$, $C$, and $D$ be as in Theorem \ref{thm:PullbackUF}, and using that $\alpha\beta = p^{2k-1}$, we deduce that 
\[
(\alpha-p^{k-1}) A_{\alpha} = A -(p^k+\beta), \quad (\alpha-p^{k-1})B_{\alpha} = B, \quad  (\alpha-p^{k-1})C_{\alpha} = C, \quad (\alpha-p^{k-1})D_{\alpha} = D + \beta p^k.
\]
Substituting $A$, $B$, $C$, and $D$ by their expression as in Theorem \ref{thm:PullbackUF}, one eventually gets the claimed formulae (see Corollary \ref{cor:Falpha-appendix} in the Appendix for details).
\end{proof}

The rest of this section is devoted to the proof of Theorem \ref{thm:PullbackUF}. The key ingredient is the following simple and yet striking relation between the operators $U$ and $U\times U$ via pullback.

\begin{proposition}\label{prop:solve it!}
    Let $F \in S_{k+1}^{(2)}(N)$ be as above. Then  we have the following equality:
    \[
        (-1)^{(p-1)/2} U\times U
        \varpi(F \otimes \varepsilon_p) = U\times U 
        \varpi(F) - \varpi(UF).
    \]
\end{proposition}
\begin{proof}
    Using the $q$-expansion, let us focus our attention on the $(n,m)$-th coefficient of $\varpi(F\otimes\varepsilon_p)$, which is
    \[
        \alpha_{n,m}\bigl( \varpi(F\otimes \varepsilon_p) )  = \sum_{r^2< 4mn} A(n,r,m) \varepsilon_p(4nm-r^2).
    \]
    By applying $U\times U$, the $(n,m)$-th coefficient becomes   
    \begin{align*}
        \alpha_{n,m}\bigl( U\times U \varpi(F\otimes \varepsilon_p) )
        &= 
        \sum_{r^2< 4mnp^2} A(pn,r,pm) \varepsilon_p(-r^2) 
        =
        (-1)^{(p-1)/2}\sum_{\substack{r^2< 4mnp^2 \\ p\nmid r}} A(pn,r,pm)
        = \\
        &= (-1)^{(p-1)/2}\sum_{r^2< 4mnp^2} A(pn,r,pm) - (-1)^{(p-1)/2}\sum_{\substack{r^2< 4mnp^2 \\ p\mid r}} A(pn,r,pm)
        = \\
        &= (-1)^{(p-1)/2}\sum_{r^2< 4mnp^2} A(pn,r,pm) - (-1)^{(p-1)/2}\sum_{r^2< 4mn}     A(pn,pr,pm).
    \end{align*}
    We identify in the last expression the $(n,m)$-th coefficients of $U\times U\varpi(F)$ and of $\varpi(UF)$, respectively. Therefore, we deduce that
    \begin{equation*}
        U\times U \varpi(F\otimes \varepsilon_p) = (-1)^{(p-1)/2}U\times U \varpi(F) - (-1)^{(p-1)/2}\varpi(UF),
    \end{equation*}
    which multiplying by $(-1)^{(p-1)/2}$ yields the result.
\end{proof}

An immediate consequence of this proposition is the following identity, which avoids the presence of $F\otimes \varepsilon_p$, and is better suited for determining $\varpi(UF)$.

\begin{corollary}\label{cor:UF}
Let $F=\SK_N(h) \in S_{k+1}^{(2)}(N)$ be a Saito--Kurokawa lift. Then we have
\[
    (U\times U - p^{k-1}) \varpi(UF) = (p^k-p^{k-1}+a_p)U\times U\varpi(F) - p^k(p^{k-1}+a_p)\varpi(F) + p^{3k-1}V\times V\varpi(F),
\]
where $a_p = a_p(f)$.
\end{corollary}
\begin{proof}
This follows by combining Corollary \ref{cor:UFrelation} and Proposition \ref{prop:solve it!}. Indeed, Corollary \ref{cor:UFrelation} tells us that
\[
    UF = (p^k + a_p) F - (p^{2k-1}+p^ka_p) VF + p^{3k-1}V^2F - (-1)^{(p-1)/2}p^{k-1} F\otimes \varepsilon_p.
\]
Applying pullback and then $U\times U$ yields the relation
\[
    U\times U \varpi(UF) = (p^k + a_p)U\times U\varpi(F) - (p^{2k-1}+p^ka_p) \varpi(F) + p^{3k-1}V\times V\varpi(F) - (-1)^{(p-1)/2}p^{k-1} U\times U \varpi(F\otimes \varepsilon_p),
\]
where we have used $\varpi\circ V = V\times V \varpi$. And now, Proposition \ref{prop:solve it!} tells us that
\[
(-1)^{(p-1)/2}p^{k-1} U\times U \varpi(F\otimes \varepsilon_p) = U\times U p^{k-1}\varpi(F) - p^{k-1}\varpi(UF),
\]
and hence we obtain
\[
(U\times U - p^{k-1})\varpi(UF) = (p^k-p^{k-1}+a_p)U\times U\varpi(F)-p^k(p^{k-1}+a_p)\varpi(F)+p^{3k-1}V\times V\varpi(F),
\]
which is the claimed relation.
\end{proof}

Fix now a normalised newform $\phi \in S_{k+1}^{\new}(N)$. The $\phi \times \phi$-component of $S_{k+1}(Np) \otimes S_{k+1}(Np)$ is spanned by the forms $\phi \times \phi$, $\phi \times V\phi$, $V\phi\times \phi$, and $V\phi \times \phi$. The following lemma, describing the action of the operator $\Xi:= U\times U - p^{k-1}$ on such a component, will be of good use for our computation below.

\begin{lemma} \label{lemma:Ximatrix}
    Let $\phi\in S_{k+1}(N)$ be a normalised eigenform. The matrix of the operator $\Xi := U\times U - p^{k-1}$ acting on the $\phi\times\phi$ component of the space $S_{k+1}(Np)\otimes S_{k+1}(Np)$, relative to the basis
    \[
        \phi\times \phi, \quad \phi\times V\phi, \quad V\phi\times \phi, \quad V\phi\times V\phi,
    \]
    is given by:
    \[
        M_{\Xi} =
        \begin{pmatrix}
            a_p(\phi)^2 -p^{k-1}  &-a_p(\phi)p^k  &-a_p(\phi)p^k  &p^{2k} \\
            a_p(\phi)  &-p^{k-1}  &-p^k  &0 \\
            a_p(\phi)  &-p^k  &-p^{k-1}  &0 \\
            1  &0  &0  &-p^{k-1}
        \end{pmatrix}.
    \]
\end{lemma}
\begin{proof}
    We know that $UV\phi=\phi$ and, from equation \eqref{eqn:TUV}, that $U\phi = a_p(\phi)\phi - p^kV\phi$. From this, we can easily describe the action of $U\times U$ on the elements of the basis for $S_{k+1}(Np)[\phi]\times S_{k+1}(Np)[\phi]$. The result follows using that $M_\Xi = M_{U\times U} - p^{k-1} \mathrm{Id}_4$.
    \comm{
    \begin{align*}
        U\times U \,  \, (\phi\times \phi) \,  \,  &=U\phi\times U\phi = a_p(\phi)^2\phi\times\phi - a_p(\phi)p^k \phi\times V\phi - a_p(\phi)p^kV\phi\times \phi + p^{2k}V\phi\times V\phi,\\
        U\times U \, (V\phi\times \phi) \,  &= \phi\times U\phi \, \,  = a_p(\phi)\phi\times\phi -p^k\phi\times V\phi, \\
        U\times U \, (\phi\times V\phi) \,  &= U\phi\times \phi \,  \,  = a_p(\phi)\phi\times\phi -p^kV\phi\times \phi, \\
        U\times U (V\phi\times V\phi) &= \phi\times\phi,
    \end{align*}
    }
\end{proof}

\begin{proof}[Proof of Theorem \ref{thm:PullbackUF}]
First of all, it is clear from Corollary \ref{cor:UF} that $\varpi(UF) = 0$ if $\varpi(F) = 0$, which is coherent with the statement. Thus we may assume that $\lambda_{\phi} \neq 0$ from now on, and write $\varpi(UF)$ in terms of unknown coefficients $A$, $B$, $C$, $D$ as in the statement, to be determined. 

The proof boils down to restricting the identity in Corollary \ref{cor:UF} to the $\phi\times \phi$-component, and solving the resulting equation for $\varpi_{\phi}(UF)$. We have 
\begin{equation}\label{eq:Xi[phi]}
\Xi\varpi_\phi(UF) = (p^k-p^{k-1}+a_p)U\times U\varpi_\phi(F) - p^k(p^{k-1}+a_p)\varpi_\phi(F) + p^{3k-1}V\times V\varpi_\phi(F),
\end{equation}
where $\Xi = U\times U - p^{k-1}$. On the right hand side of \eqref{eq:Xi[phi]}, expanding $U\times U\phi\times \phi$ we find
\[
\lambda_{\phi}(p^k-p^{k-1}+a_p)\left( a_p(\phi)^2 \phi\times \phi - a_p(\phi) p^k \phi\times V\phi - a_p(\phi)p^k V\phi\times \phi + p^{2k}V\phi\times V\phi\right) 
\]
\[
- \lambda_{\phi}p^k(p^{k-1}+a_p)\phi\times \phi + \lambda_{\phi}p^{3k-1} V\phi\times V\phi. 
\]
And the left hand side of \eqref{eq:Xi[phi]}, in matrix notation, reads
\[
\lambda_{\phi} \begin{pmatrix} A & B & C & D \end{pmatrix} M_{\Xi},
\]
where $M_{\Xi}$ is as in lemma \ref{lemma:Ximatrix}. Equating these expressions yields a linear system of four equations with four unknowns, whose solution is the one given in the statement. We give the complete computation in Proposition \ref{prop:PullbackUF-appendix} in the Appendix.
\end{proof}

\section{\texorpdfstring{$p$}{p}-adic families of modular forms} \label{sec:p-adic families}

Let $p$ be an odd prime number. Fix algebraic closures $\barQ$ and $\barQp$ of $\Q$ and $\Qp$, respectively, and field embeddings $\barQ \hookrightarrow \C$, $\barQ \hookrightarrow \barQp$. Fix also an isomorphism of fields $\C \simeq \C_p$, compatible with the previous embeddings, meaning that the following diagram commutes:
\[
    \xymatrix{
        \barQ \ar@{^{(}->}[r] \ar@{^{(}->}[d] &\barQp \ar@{^{(}->}[d] \\
        \C \ar[r]^\simeq &\Cp.
    }
\]
Let $\mathcal O$ be the ring of integers of a  finite extension of $\Q_p$ and let $\Gamma := 1 + p\Z_p$ be the group of principal units in $\Z_p$. We write $\Lambda = \Lambda_{\mathcal O} := \mathcal O[[\Gamma]]$ for the usual Iwasawa algebra over $\mathcal O$, $\mathfrak L$ for its fraction field, and consider the space
\[
    \mathcal W := \Hom_{\cO-\cont}(\Lambda, \Cp)
\]
of continuous $\mathcal O$-algebra homomorphisms from $\Lambda$ to $\Cp$. Elements $\mathbf a \in \Lambda$ can be seen as functions on $\mathcal W$ through evaluation at $\mathbf a$, i.e. by setting $\mathbf a(\kappa) := \kappa(\mathbf a)$, and elements $\mathbf a\in \mathfrak L$ can be seen as meromorphic functions having finitely many poles. The set $\mathcal W$ is endowed with the analytic structure induced from the natural identification 
\begin{equation}\label{XLambda-characters}
    \mathcal W \simeq \mathrm{Hom}_\cont(\Gamma, \C_p^{\times})
\end{equation}
between $\mathcal W$ and the group of continuous characters $\kappa: \Gamma \to \C_p^{\times}$. A character $\kappa: \Gamma \to \C_p^{\times}$ is called {\em arithmetic} (resp. {\em classical}) if there exists an integer $k \geq 0$ such that $\kappa(t) = t^k$ for all $t$ {\em sufficiently close} to $1$ in $\Gamma$ (resp. for all $t \in \Gamma$). A point $\kappa \in \mathcal W$ is said to be {\em arithmetic} (resp. {\em classical}) if the associated character of $\Gamma$ under \eqref{XLambda-characters} is arithmetic (resp. classical). If this is the case, we refer to the integer $k$ as the {\em weight} of $\kappa$, and we write $k = \mathrm{wt}(\kappa)$. In this note, we will restrict to classical points; write $\mathcal W^{\cl} \subset \mathcal W$ for the subset of all such points.

If $\mathcal R$ is a finite flat $\Lambda$-algebra, then we write 
\[
    \mathcal W(\mathcal R) := \mathrm{Hom}_{\cO-\cont}(\mathcal R, \Cp)
\]
for the set of continuous $\mathcal O$-algebra homomorphisms from $\mathcal R$ to $\Cp$, to which we also refer as `points of $\mathcal R$'. Elements of $\mathcal K := \mathrm{Frac}(\mathcal R)$ can be regarded as meromorphic functions on $\mathcal W(\mathcal R)$. The restriction to $\Lambda$ (via the structure morphism $\Lambda \to \mathcal R$) induces a surjective finite-to-one map 
\[
    \pi: \mathcal W(\mathcal R) \, \longrightarrow \, \mathcal W.
\]
One can define analytic charts around all points $\kappa$ of $\mathcal W(\mathcal R)$ which are unramified over $\Lambda$, by building sections $S_\kappa$ of the map $\pi$, so that $\mathcal X(\mathcal R)$ inherits the structure of rigid analytic cover of $\mathcal W$. A function $f: \mathcal U \subseteq \mathcal W(\mathcal R) \to \Cp$ defined on an analytic neighborhood of $\kappa$ is analytic if so is $f \circ S_{\kappa}$. The evaluation at an element $\mathbf a \in \mathcal R$ yields a function $\mathbf a: \mathcal W(\mathcal R) \to \Cp$, $\mathbf a(\kappa) := \kappa(\mathbf a)$, which is analytic at every unramified point of $\mathcal W (\mathcal R)$. A point $\kappa \in \mathcal W(\mathcal R)$ is said to be classical if the point $\pi(\kappa) \in \mathcal W$ is classical. We write $\mathcal W^{\cl}(\mathcal R)\subseteq \mathcal W(\mathcal R)$ for the subset of classical points in $\mathcal W(\mathcal R)$, and if $\kappa \in \mathcal W^{\cl}(\mathcal R)$ we continue to use the notation $\mathrm{wt}(\kappa) = \mathrm{wt}(\pi(\kappa))$ for the weight of $\kappa$.

At last, we want to remark an example that will be of particular interest for us: if $\mathcal R := \mathcal R_1 \otimes_\Lambda \mathcal R_2$ then elements of $\mathcal K_1\otimes_\Lambda \mathcal K_2$ can be regarded as meromorphic functions on the fiber product $\mathcal W(\mathcal R_1) \times_{\mathcal W} \mathcal W(\mathcal R_2)$, having poles at finitely many points.

\subsection{Hida theory for \texorpdfstring{$\GL_2$}{$\GL_2$} }\label{subsec:HidaGL2}

Let $N\geq 1$ be an integer, assume that $p\nmid N$, and fix an integer $k\geq 2$. The action of the operator $U_p$ on the space $S_{k}(Np,\barQp)$ yields Hida's {\em ordinary projector}
\begin{equation}\label{def:eord}
    e_{\ord} := \lim_{n\to\infty} U_p^{n!},
\end{equation}
which is an idempotent in $\mathrm{End}(S_{k}(Np,\barQp))$. The subspace of ordinary cusp forms in $S_{k}(Np,\barQp)$ will be denoted  
\[
    S_{k}^{\ord}(Np,\barQp) := e_{\ord}S_{k}(Np,\barQp) \subseteq S_{k}(Np,\barQp).
\]
It is well-known that the dimension of $S_{k}^{\ord}(Np,\barQp)$ is finite and independent of $k$ for $k \geq 2$ (cf. \cite[\S 7.2]{Hid93}). 
If $\varphi\in S_k(Np,\barQp)$ is an eigenform for $U_p$, then either $e_\ord\varphi = \varphi$ or $e_\ord\varphi = 0$. In the first case, $\varphi$ is either
\begin{itemize}
    \item new at $p$, which can only happen when $k=2$ (cf. \cite{Miyake}) and we will exclude this case from now on, or
    \item the {\em ordinary $p$-stabilisation} of a normalised eigenform $\phi \in S_{k}(N,\bar{\Q}_p)$: that is to say, $\varphi = \phi_{\alpha} = (q-\beta V)\phi$, where $\alpha$ and $\beta$ are the roots of the $p$-th Hecke polynomial of $g$ (see section \ref{sec:ClassicalBasics}), labelled so that $|\alpha|_p \geq |\beta|_p$, hence $\alpha$ is a $p$-adic unit. We will say that such a $\phi$ is {\em ordinary} at $p$, or just ordinary.
\end{itemize}

If $\phi \in S_k^{\new}(N,\barQp)$ is ordinary at $p$, the action of $\eord$ on the $\phi$-isotypic subspace $S_k(Np,\barQp)[\phi]$ of $S_k(Np,\barQp)$ is characterised by the fact that $\eord \phi_{\alpha} = \phi_{\alpha}$ and $\eord \phi_{\beta} = 0$.

Let now $\mathcal R$ be a finite flat $\Lambda$-algebra. We denote by $\mathbf S(N,\mathcal R)$ the space of $\Lambda$-adic cusp forms of tame level $\Gamma_0(N)$ over $\mathcal R$, namely the $\mathcal R$-module of formal power series
\[
    \pmb{\phi} = \sum_{n\geq 1} \mathbf a_n(\pmb{\phi}) q^n \in \mathcal R[[q]]
\]
such that $\pmb{\phi}(\kappa)$ is the $q$-expansion of a classical modular form in $S_k(\Gamma_0(N) \cap \Gamma_1(p))$, which we still denote $\pmb{\phi}(\kappa)$, for every classical point $\kappa \in \mathcal W^{\cl}(\mathcal R)$ of weight $\mathrm{wt}(\kappa) = k-2\geq 0$. By restricting to classical points whose weights are all contained in a suitable residue class modulo $p-1$, we may assume that all the $\pmb{\phi}(\kappa)\in S_k(Np)$. 

A $\Lambda$-adic cusp form $\pmb{\phi} \in \mathbf S(N,\mathcal R)$ is said to be {\em ordinary} if for every classical point $\kappa \in \mathcal W^{\cl}(\mathcal R)$ the specialisation $\pmb{\phi}(\kappa)$ is an ordinary cusp form, i.e. belongs to $S_k^{\mathrm{ord}}(Np)$. Write $\mathbf S^{\ord}(N,\mathcal R)$ for the subspace of ordinary $\Lambda$-adic forms in $\mathbf S(N,\mathcal R)$. By the work of Hida, there exists a unique idempotent on $\mathbf S(N,\mathcal R)$, which by abuse of notation we continue to denote $e_{\ord}$, such that 
\[
    \mathbf S^{\ord}(N,\mathcal R) =  e_{\ord}\mathbf S(N,\mathcal R).
\]

For the purpose of this note, we may restrict ourselves to the following definition of Hida family as a family of ordinary $p$-stabilised newforms:

\begin{definition}\label{def:Hidafamily}
    A {\em Hida family} of tame level $N$ is a quadruple $(\mathcal R_{\mathbf f}, \cU_{\mathbf f}, \cU_{\mathbf f}^\mathrm{cl}, \mathbf f)$ where:
    \begin{itemize}
        \item $\mathcal R_{\mathbf f}$ is a finite flat integral domain extension of $\Lambda$;
        \item $\mathcal U_{\mathbf f} \subset \mathcal W(\mathcal R_{\mathbf f})$ is an open subset for the rigid analytic topology; 
        \item $\cU_{\mathbf f}^{\mathrm{cl}} \subset \mathcal U_{\mathbf f} \cap \mathcal W^{\mathrm{cl}}(\mathcal R_{\mathbf f})$ is a dense subset of $\mathcal U_{\mathbf f}$ whose weights are contained in a single residue class $k_0-2$ modulo $p-1$; 
        \item $\mathbf f \in \mathbf S^{\ord}(N,\mathcal R_{\mathbf f})$ is an ordinary $\Lambda$-adic cusp form over $\mathcal R_{\mathbf f}$ such that for all $\kappa\in \cU^{\mathrm{cl}}_{\mathbf f}$ of weight $k-2 > 0$, 
        \[
            \mathbf f(\kappa) \in S_k^{\mathrm{ord}}(Np,\barQp)
        \]
        is the ordinary $p$-stabilisation of a normalised Hecke eigenform $f_\kappa \in S_k^{\new}(N,\barQp)$.
\end{itemize}
\end{definition}

In order to simplify notation, for a Hida family as in the definition we will abbreviate $\mathcal W_{\hf} := \mathcal W(\mathcal R_{\hf})$, and similarly $\mathcal W_{\hf}^{\cl} := \mathcal W^{\cl}(\mathcal R_{\hf})$. Given an ordinary $p$-stabilised newform $f_{0} \in S_{k_0}(Np,\barQp)$ with $k_0\geq 2$, Hida's theory ensures the existence of a unique Hida family $\mathbf f$ as in the definition passing through $f_0$ at a distinguished classical point $\kappa_0$ of weight $k_0-2$ (cf. \cite{Hid86}). 

\begin{remark} 
When $k=\mathrm{wt}(\kappa)+2=2$, the form $\mathbf f(\kappa)$ can be either old or new at $p$. Only in the first case, $\mathbf f(\kappa)$ will be the $p$-stabilisation of a weight $2$ newform of level $N$.
\end{remark}

\begin{lemma}\label{lemma:Jfg}
    Let $\mathbf g$ be a Hida family as above and let $\mathcal K_\hg := \mathrm{Frac}(\mathcal R_\hg)$. There exists a unique element $\mathcal J_{\mathbf g} :\mathbf S^{\ord}(N,\mathcal R) \to \mathcal R \otimes_{\Lambda} \mathcal K_{\mathbf g}$ such that for every $\pmb{\phi} \in \mathbf S^{\ord}(N,\mathcal R)$ 
    \[
        \mathcal J_{\mathbf g}(\pmb{\phi})(\kappa,\kappa') = \frac{\langle  \pmb{\phi}(\kappa),\mathbf g(\kappa')\rangle}
        {\langle \mathbf g(\kappa'),\mathbf g(\kappa')\rangle}
    \]
    for all $(\kappa,\kappa') \in \mathcal U_{\pmb{\phi}}^{\cl} \times_{\mathcal W^{\cl}} \mathcal U_{\mathbf g}^{\cl}$.
\end{lemma}
\begin{proof}
This is a standard argument. Indeed, $\mathbf S^{\ord}(N,\mathcal R) \otimes \mathcal K_{\mathbf g}$ is a finite-dimensional vector space over $\mathcal K_{\mathbf g}$, and Hida theory provides an idempotent $e_{\mathbf g}$ of the $\Lambda$-adic Hecke algebra over $\mathcal K_{\mathbf g}$ whose specialization at a classical weight $\kappa'$ is the usual idempotent $e_{\mathbf g(\kappa')}$ associated with the classical cusp form $\mathbf g(\kappa')$.  Therefore, the desired element $\mathcal J_{\mathbf g}$ is defined by requiring that $e_{\mathbf g}(\pmb{\phi})= \mathcal J_{\mathbf g}(\pmb{\phi}) \cdot \mathbf g$ for each $\pmb{\phi} \in \mathbf S^{\ord}(N,\mathcal R)$ (a detailed proof in a slightly more general setting can be found in \cite[Lemma 2.19]{DR14}).
\end{proof}

\begin{remark}
    Notice that if $\hg(\kappa')$ is the ordinary $p$-stabilisation $g_\alpha$ of a modular form of level $N$, then $\mathcal J_\hg(\pmb{\phi})(\kappa,\kappa')$ is the coefficient of $\pmb{\phi}(\kappa)$ in the $g_\alpha$-direction. 
\end{remark}

We conclude the section with the following result about the existence of a $p$-adic $L$-function associated with a Hida family. In \cite{GS93}, Greenberg and Stevens defined a two-variable $p$-adic $L$-function interpolating the special values of the completed $L$-functions
\[
    \Lambda(f_\kappa,\psi,s) := L_\infty(f_\kappa,\psi,s) \cdot L(f_\kappa,\psi,s), \qquad \text{with } \quad L_\infty(f_\kappa,\psi,s) = \Gamma_\C(s) = 2(2\pi)^{-s} \Gamma(s),
\]
where $f_\kappa$ are the specialisations of a Hida family and $\psi$ is a Dirichlet character. In particular, we define their algebraic parts to be
\[
    \Lambda(f_\kappa,\psi,s)^\mathrm{alg} := \frac{\Lambda(f_\kappa,\psi,s) }{ \Omega_{f_\kappa}^{\mathrm{sgn}(\psi) (-1)^{s-1}} }.
\]

\begin{theorem}[Mazur, Kitagawa, Greenberg--Stevens] \label{thm:mainGS}
Let $\hf \in \mathcal R_\hf[[q]]$ be a Hida family and let $\psi$ be a Dirichlet character such that $p^m || \mathrm{cond}(\psi)$. Then there exists a unique element $\mathcal L_p(\mathbf f, \psi) \in \mathcal R_\hf \otimes \Lambda$ satisfying the following interpolation property: for all $(\kappa,\sigma) \in \mathcal U_\hf^\cl \times \mathcal W^\cl$ of weight $(2k-2, s)$, where $0<s<2k$, we have 
\[
    \mathcal L_p(\mathbf f, \psi)(\kappa,\sigma) = \Omega_{\kappa}^{\mathrm{sgn}(\psi)} \cdot \mathcal E(f_\kappa, \psi,s)  \cdot \alpha_{f_\kappa}^{1-m} \cdot \frac{ c^s \cdot \psi(-1)\omega(-1)^{s-1} }{ i^{-k} \mathfrak g(\bar\psi\omega^{s-1}) } \cdot \Lambda( f_\kappa,\bar{\psi}\omega^{s-1},s)^\mathrm{alg},
\]
where $c= \mathrm{cond}(\bar\psi\omega^{s-1})$, $\Omega_\kappa^\pm$ are the $p$-adic periods defined as in \cite{GS93}, and
\[
    \mathcal E(f_\kappa, \psi,s) = 
    \left(1-\frac{\psi \omega^{1-s}(p)p^{s-1}}{\alpha_{f_\kappa}}\right) \left(1-\frac{\bar\psi \omega^{s-1}\chi_0(p)p^{2k-1-s}}{\alpha_{f_\kappa}}\right).
\]
\end{theorem}

\begin{remark}
    This is essentially the construction of \cite{GS93} multiplied by $\mathbf a_p := a_p(\hf) \in \mathcal R_\hf$. This different normalisation is chosen to obtain cleaner formulas later.
\end{remark}

\subsection{Tensor products of Hida families}\label{subsec:HidaGL2xGL2}

Let $N$ and $p$ be as before, and fix an integer $k \geq 1$. Recall Hida's ordinary projector $e_{\mathrm{ord}}$ acting on $S_{k+1}(Np,\barQp)$, as defined in \eqref{def:eord}. Consider now the tensor space $S_{k+1}(Np,\barQp)\otimes S_{k+1}(Np,\barQp)$, whose Hecke algebra contains the elements of the form $T_1 \otimes T_2$ with $T_1$, $T_2$ in the usual Hecke algebra of level $N$ and weight $k+1$. In particular, we can form the limit of $(U\otimes U)^{n!}$, which yields the idempotent $e_{\mathrm{ord}} \otimes e_{\mathrm{ord}}$, and the corresponding projector 
\[
e_{\mathrm{ord}} \otimes e_{\mathrm{ord}}: S_{k+1}(Np,\chi)\otimes S_{k+1}(Np,\chi) \to S_{k+1}^{\ord}(Np) \otimes S_{k+1}^{\ord}(Np).
\]
By a slight abuse of notation, we will write $e_{\mathrm{ord}}$ for $e_{\mathrm{ord}} \otimes e_{\mathrm{ord}}$, and it will be clear from the context which projector we are using.

Let $\phi \in S_{k+1}(N,\barQp)$ be a normalised ordinary eigenform, and let $\phi_{\alpha}$, $\phi_{\beta}$ denote the $p$-stabilisations of $\phi$ with respect to the roots $\alpha=\alpha_{\phi}$ and $\beta=\beta_{\phi}$ of its $p$-th Hecke polynomial. If $\alpha$ is the unit root, similarly as in the $\GL_2$-case now the ordinary projector on the $\phi \times \phi$-isotypic component $\left(S_{k+1}(Np,\barQp)\otimes S_{k+1}(Np,\barQp)\right)[\phi]$ is characterised by the fact that 
\[
    \eord \phi_{\alpha}\times \phi_{\alpha} = \phi_{\alpha}\times \phi_{\alpha}, \quad \eord \phi_{\alpha}\times \phi_{\beta} = \eord \phi_{\beta}\times \phi_{\alpha} = \eord \phi_{\beta}\times \phi_{\beta} = 0.
\]
In particular, $e_{\ord}\left(\left(S_{k+1}(Np,\barQp)\otimes S_{k+1}(Np,\barQp) \right)[\phi]\right) = \langle \phi_{\alpha} \times \phi_{\alpha} \rangle.$

Let now $\mathcal R$ be a finite flat $\Lambda$-algebra, and $\mathbf S(N,\mathcal R)$ denote the space of $\Lambda$-adic cusp forms of level $\Gamma_0(Np)$ over $\mathcal R$ as above. We can then consider the $\mathcal R$-module $\mathbf S(N,\mathcal R) \otimes_{\mathcal R} \mathbf S(N,\mathcal R)$, equipped with the corresponding ordinary projector 
\[
    e_{\ord} \otimes e_{\ord}: \mathbf S(N,\mathcal R) \otimes_{\mathcal R} \mathbf S(N,\mathcal R) \, \longrightarrow \, \mathbf S^{\ord}(N,\mathcal R) \otimes_{\mathcal R} \mathbf S^{\ord}(N,\mathcal R).
\]
The $\GL_2\times \GL_2$ analogue of Lemma \ref{lemma:Jfg} is given by the following:

\begin{lemma}\label{lemma:JFg}
    Let $\mathbf g$ be a Hida family and write $\mathcal K_\hg := \mathrm{Frac}(\mathcal R_\hg)$. Then there exists a unique functional $\mathcal J_{\mathbf g\times \mathbf g} :  \mathbf S^{\ord}(N,\mathcal R) \otimes_{\mathcal R}  \mathbf 
    S^{\ord}(N,\mathcal R) \to \mathcal R \otimes_\Lambda \mathcal K_\hg$ such that for every $\mathcal F \in  \mathbf S^{\ord}(N,\mathcal R)\otimes_{\mathcal R}  \mathbf S^{\ord}(N,\mathcal R)$ we have
    \[
    \mathcal J_{\mathbf g\times \mathbf g}(\mathcal F)(\kappa,\lambda) = \frac{\langle \mathcal F(\kappa), \mathbf g(\lambda)\times \mathbf g(\lambda)\rangle}{\langle \mathbf g(\lambda)\times \mathbf g(\lambda), \mathbf g(\lambda)\times \mathbf g(\lambda)\rangle}
    \]
    for all $(\kappa,\lambda) \in \mathcal U_{\mathcal F}^{\cl} \times_{\mathcal W^{\cl}} \mathcal U_{\mathbf g}^{\cl}$.
\end{lemma}

\subsection{\texorpdfstring{$p$}{p}-adic families of half-integral weight modular forms: a  \texorpdfstring{$\Lambda$}{Lambda}-adic \texorpdfstring{$\d$}{d}-th Shintani lifting}\label{sec:padicShintani}

The theory of $p$-adic families of half-integral weight modular forms is not yet developed in full generality and most known results are based on the Shintani lifting construction. We will describe a construction in a context broad enough to suite our needs. For this reason, we fix a Hida family $(\mathcal R_{\mathbf f}, \mathcal U_{\mathbf f},\mathcal U_{\mathbf f}^{\mathrm{cl}},\mathbf f)$ of tame level $N$ (and trivial tame character) as above. Recall that by definition of Hida family, there exists an integer $k_0$ (only determined modulo $(p-1)/2$) such that every classical point $\kappa \in \mathcal U_{\mathbf f}^{\mathrm{cl}}$ has weight $2k-2$ with $2k \equiv 2k_0 \pmod{p-1}$. We define 
\[
    \widetilde{\mathcal R}_{\mathbf f} := \mathcal R_{\mathbf f} \otimes_{\Lambda,\sigma} \Lambda, 
\]
where $\sigma: \Lambda \to \Lambda$ is the $\mathcal O$-algebra isomorphism induced by $[t] \mapsto [t^2]$ on $\Gamma = 1+p\Z_p$, and write $\widetilde{\mathcal W}_{\hf} := \mathcal W(\widetilde{\mathcal R}_{\mathbf f})$ for the associated weight space. We equip $\widetilde{\mathcal R}_{\mathbf f}$ with the structure of $\Lambda$-algebra via the map $\lambda \mapsto 1 \otimes \lambda$. The natural homomorphism 
\[
    \mathcal R_{\mathbf f} \, \longrightarrow \, \widetilde{\mathcal R}_{\mathbf f}, \quad \alpha \, \longmapsto \, \alpha \otimes 1
\]
is an isomorphism of $\mathcal O$-algebras, but it is {\em not} a homomorphism of $\Lambda$-algebras. Indeed, this is reflected in the fact that the induced map 
\begin{equation} \label{eqn:Weights Pullback}
    \pi: \widetilde{\mathcal W}_{\hf} \, \longrightarrow \, \mathcal W_{\mathbf f}
\end{equation}
on weight spaces doubles the weights: if $\tilde{\kappa} \in \widetilde{\mathcal W}^{\cl}_{\hf} := \mathcal W^{\cl}(\widetilde{\mathcal R}_{\mathbf f})$ has weight $k$, then $\pi(\tilde{\kappa}) \in \mathcal W_{\hf}^{\cl}$ has weight $2k$.

Fix once and for all a solution $r_0$ of the congruence $2x \equiv 2k_0 \pmod{p-1}$, write $\widetilde{\mathcal U}_{\mathbf f} := \pi^{-1}(\mathcal U_{\mathbf f})$, $\widetilde{\mathcal U}_{\mathbf f}^{\mathrm{cl}} := \pi^{-1}(\mathcal U_{\mathbf f}^{\mathrm{cl}})$, and let $\widetilde{\mathcal U}_{\mathbf f}^{\mathrm{cl}}(r_0)$ denote the subset of classical points in $\widetilde{\mathcal U}_{\mathbf f}^{\mathrm{cl}}$ whose weights are congruent to $r_0-1$ modulo $p-1$. Observe that this is only {\em half} of $\widetilde{\mathcal U}_{\mathbf f}^{\mathrm{cl}}$, the {\em other half} consisting of points whose weights are congruent to $r_0 + (p-1)/2$ modulo $p-1$. 

\begin{theorem} \label{thm:mainCdVP21}
    Fix a discriminant $\d$ with $p\mid \d$ and $(-1)^{r_0}\d>0$. There exists a unique element
    \[
        \pmb\Theta = \sum_{m\geq 1} \mathbf c(m) q^m \in \widetilde{\mathcal R}_\hf[[q]],
    \]
    such that for all $\tilde\kappa\in \widetilde{\mathcal U}_\hf^\cl(r_0)$ of weight $k-1$ we have
    \[
        \pmb\Theta(\tilde\kappa) = \Omega_{\kappa}^- \cdot (-1)^{[k/2]}2^{-k} \cdot \theta_{k,Np,\d}^\mathrm{alg}(\hf(\kappa)),
    \]
    where $\kappa := \pi(\tilde{\kappa}) \in \mathcal U_{\hf}^{\cl}$ and $\theta^{\mathrm{alg}}_{k,Np,\d}(\hf(\kappa)) = \theta_{k,Np,\d}(\hf(\kappa))/\Omega_{f_\kappa}^-$ denotes the {\em algebraic} $\d$-th Shintani lifting as explained in \eqref{eqn:algebraicShintani}. Here, the $p$-adic periods $\{\Omega_{\kappa}^-\}_{\kappa \in \mathcal U_{\hf}^{\cl}}$ are as in Theorem \ref{thm:mainGS}.
\end{theorem}

\begin{proof}
    The collection of $p$-adic periods, the construction of the element, and its interpolation property are described, respectively, in Corollary 4.6, equation (31) and Theorem 5.9 of \cite{CdVP21}.
\end{proof}

The element $\pmb{\Theta} \equiv \Theta_\d^{r_0}(\hf)$ of the above theorem is attached to the Hida family $\hf$, but it depends on both choices $\d$ and $r_0$. Anyway, we will drop them from the notation as we will see that in our setting we will only have a suitable choice for $r_0$ and that our final result is independent on the choice of $\d$. We point out that the specialisation of $\pmb{\Theta}$ at a classical point $\tilde{\kappa} \in \widetilde{\mathcal U}_\hf^{\cl}(r_0)$ of weight $k-1$ has weight $k+1/2$. 

\begin{remark}\label{rmk:shintani-vanishing}
It may happen that $\pmb\Theta$ vanishes identically, but as soon as there exists a classical point $\tilde{\kappa}_0 \in \widetilde{\mathcal U}_{\mathbf f}^{\cl}(r_0)$ with $\theta_{k,Np,\d}(\mathbf f(\kappa_0)) \neq 0$, where $\kappa_0 = \pi(\tilde{\kappa}_0)$, one can choose the $p$-adic periods $\Omega_{\kappa}^-$ to be non-vanishing in a neighborhood of $\kappa_0$ and $\pmb\Theta$ will not vanish identically in that neighborhood.
\end{remark}

\begin{proposition} \label{prop:metaplecticGS}
    With the above notation, there exists a unique element $\widetilde{\mathcal L}_p(\mathbf f, \psi) \in \widetilde{\mathcal R}_\hf$ satisfying the following interpolation property. For every $\tilde\kappa\in \widetilde{\mathcal U}^\cl_\hf(r_0)$ of weight $k-1$, setting $\kappa = \pi(\tilde\kappa)$ we have
    \[
        \widetilde{\mathcal L}_p(\mathbf f, \psi)(\tilde\kappa) = \Omega_{\kappa}^{\mathrm{sgn}(\psi)} \cdot \mathcal E(f_\kappa, \psi,k) \cdot \alpha_{f_\kappa}^{1-m} \cdot \frac{c^k \cdot \psi(-1)\omega(-1)^{k-1}}{ i^{k}\mathfrak g(\bar\psi\omega^{k-1})}  \Lambda( f_\kappa,\bar{\psi}\omega^{k-1},k)^\mathrm{alg},
    \]
    where $c= \mathrm{cond}(\bar\psi\omega^{k-1})$, the $p$-adic periods $\Omega_{\kappa}^{\pm}$ are as in Theorem \ref{thm:mainGS}, and
    \[
        \mathcal E(f_\kappa, \psi,k) = 
        \left(1-\frac{\psi \omega^{1-k}(p) p^{k-1}}{\alpha_{f_\kappa}}\right) \left(1-\frac{\bar\psi \omega^{k-1}p^{k-1}}{\alpha_{f_\kappa}}\right).
    \]
\end{proposition}

\begin{proof}
    This follows from Theorem \ref{thm:mainGS} taking the pullback along the map $\widetilde{\mathcal U}_\hf  \to \mathcal U_\hf \to \mathcal U_\hf \times \mathcal W$, where the second map is the extension of $\kappa \to (\kappa, \mathrm{wt}(\kappa)/2+1)$ on classical points.
\end{proof}

\subsection{\texorpdfstring{$p$}{p}-adic families of Siegel forms: a  \texorpdfstring{$\Lambda$}{Lambda}-adic Saito--Kurokawa lifting}

Let $\mathrm{Sym}_2$ denote the set of half-integral, symmetric, positive-definite two-by-two matrices, and let $\mathcal R$ be a finite flat $\Lambda$-algebra. A {\em $\Lambda$-adic Siegel form} of tame level $\Gamma_0^{(2)}(Np)$ over $\mathcal R$ is a formal power series 
\[
    \pmb{\Phi} = \sum_{B \in \mathrm{Sym}_2} \mathbf A(B) q^B \in \mathcal R[[q_1,q_2,\zeta]],
\]
where
\[
q^B = q_1^n q_2^m \zeta^r \quad \text{if } B = \begin{pmatrix}n & r/2 \\r/2 & m\end{pmatrix},
\]
whose specialisations $\pmb{\Phi}$ at classical points in $\mathcal W^{\cl}(\mathcal R)$ of weight $k\geq 0$ are the $q$-expansions of Siegel forms in $S_{k+2}^{(2)}(Np,\barQp)$. For notational purposes, we may sometimes write $\mathbf A(n,r,m)$ instead of $\mathbf A(B)$ if $B$ is as above. Write $\mathbf S^{(2)}(N,\mathcal R)$ for the space of $\Lambda$-adic Siegel forms of tame level $\Gamma_0^{(2)}(N)$ over $\mathcal R$, which is naturally an $\mathcal R$-module.

We will be interested in a very specific type of $\Lambda$-adic Siegel forms, namely $\Lambda$-adic Saito--Kurokawa lifts. 

\begin{proposition} \label{prop:LambdaSK}
    Let $\hf$ be a Hida family and let $\d$, $r_0$, $\Omega_\kappa$, and $C(k,\d)$ as in Theorem \ref{thm:mainCdVP21}. There exists a unique element
    \[
        \hSK = \sum_B \mathbf A(B)q^B \in \widetilde{\mathcal R}_{\mathbf f}[[q_1,q_2,\zeta]]
    \]
    such that for every classical point $\tilde\kappa\in \widetilde{\mathcal U}_\hf^\cl(r_0)$ of weight $k-1$, setting $\kappa = \pi(\tilde\kappa)\in \mathcal U_\hf^\cl$ we have
    \[
        \hSK(\tilde{\kappa}) = \Omega_{\kappa}^- \cdot C(k,\d)^{-1} \cdot \SK_{Np}(\theta_{k,Np,\d}^\mathrm{alg}(\mathbf f(\kappa))) \in S_{k+1}^{(2)}(Np,\barQp).
    \]
\end{proposition}

\begin{proof}
Consider the element    $\pmb\Theta =\sum_{m\geq 1} \mathbf c(m) q^m \in \widetilde{\mathcal R}_{\mathbf f}[[q]] $
of Theorem \ref{thm:mainCdVP21} and define
\[
    \mathbf A(B) = \sum_{\substack{d \mid \mathrm{gcd}(n,r,m),\\(d,Np)=1}} \omega(d)^{r_0-1} d \cdot [\langle d\rangle]  \cdot \mathbf c(\det(2B)/d^2) \in \widetilde{\mathcal R}_{\mathbf f}.
\]
The interpolation property follows directly from that of Theorem \ref{thm:mainCdVP21}.
\end{proof}

Therefore, the formal power series $\hSK$ is clearly a $\Lambda$-adic Siegel form in $\mathbf S^{(2)}(N,\widetilde{\mathcal R}_{\mathbf f})$. If we write
\[
\hSK = \sum_{B=\left(\begin{smallmatrix}n & r/2 \\ r/2 & m\end{smallmatrix}\right)} \mathbf A(n,r,m) q^B \in \widetilde{\mathcal R}_{\mathbf f}[[q_1,q_2,\zeta]],
\]
then under the natural `pullback' morphism $\varpi: \widetilde{\mathcal R}_{\mathbf f}[[q_1,q_2,\zeta]] \to \widetilde{\mathcal R}_{\mathbf f}[[q_1,q_2]]$ (defined by setting $\zeta = 1$), $\hSK$ yields the formal power series
\[
\varpi(\hSK) = \sum_{n,m\geq 1} \left( \sum_{\substack{r \in \Z,\\r^2 < 4nm}} \mathbf A(n,r,m) \right) q_1^n q_2^m.
\]
If $\tilde{\kappa} \in \widetilde{\mathcal U}_{\mathbf f}^{\cl}(r_0)$ has weight $k-1$ and $\kappa = \pi(\tilde{\kappa})$, it follows from the above proposition that
\begin{equation}\label{eq:interpolation:pullbackSK}
    \varpi(\hSK)(\tilde{\kappa}) =  \Omega_{\kappa}^- \cdot C(k,\d)^{-1} \cdot \varpi(\SK_{Np}(\theta_{k,Np,\d}^\mathrm{alg}(\mathbf f(\kappa)))) \in S_{k+1}(Np,\barQp) \otimes S_{k+1}(Np,\barQp).
\end{equation}
In fact, $\varpi(\hSK)$ belongs to $\mathbf S(N,\widetilde{\mathcal R}_{\mathbf f}) \otimes_{\widetilde{\mathcal R}_{\mathbf f}} \mathbf S(N,\widetilde{\mathcal R}_{\mathbf f})$.

\section{The \texorpdfstring{$p$}{p}-adic \texorpdfstring{$L$}{L}-function}\label{sec:padicLfunction}

Let $(\mathcal R_{\mathbf f}, \mathcal U_{\mathbf f}, \mathcal U_{\mathbf f}^\cl, \mathbf f)$ and $(\mathcal R_{\mathbf g}, \mathcal U_{\mathbf g}, \mathcal U_{\mathbf g}^\cl,\mathbf g)$ be two Hida families of tame level $N$ and trivial tame nebentype character as defined in Section \ref{subsec:HidaGL2}. As explained in Definition \ref{def:Hidafamily}, the Hida family $\hg$ determines a unique class $r_0$ modulo $p-1$, such that $\mathrm{wt}(\lambda) \equiv r_0-1 \pmod{p-1}$ for all $\lambda \in \mathcal U_{\hg}^{\cl}$. We recall from Section \ref{sec:padicShintani} that the map $\sigma:\Lambda \to \Lambda$ sending $[t]\mapsto [t^2]$ induces the $\Zp$-algebra isomorphism $\mathcal R_\hf \to \widetilde{\mathcal R}_\hf$ which induces a map on weight spaces
\[
    \pi: \widetilde{\mathcal W}_\hf \, \longrightarrow \, \mathcal W_\hf, \quad \tilde\kappa \, \longmapsto \, \kappa:=\pi(\tilde\kappa).
\]

As described in Section \ref{sec:padicShintani}, let us fix a fundamental discriminant $\d$ such that $(-1)^{r_0}\d > 0$ and $\d \equiv 0 \pmod p$, and consider its $\Lambda$-adic $\d$-th Shintani lifting $\pmb{\Theta}$ as in Theorem \ref{thm:mainCdVP21}, 
\[
    \pmb{\Theta} = \sum_{m\geq 1} \mathbf c(m) q^m  \in \widetilde{\mathcal R}_{\hf}[[q]].
\]
As customary, we will write $\Theta_{\tilde{\kappa}} \in S_{k+1/2}(N)$ for the half-integral weight modular form whose ordinary $p$-stabilisation is $\pmb{\Theta}(\tilde{\kappa})$. In order to lighten the notation, we will also write $\hh(\tilde{\kappa}) = \theta_{k,Np,\d}^\mathrm{alg}(\hf(\kappa)) \in S_{k+1/2}(Np)$ and $h_{\tilde{\kappa}} = \theta_{k,N,\d}^\mathrm{alg}(f_{\kappa}) \in S_{k+1/2}(N)$ for each $\tilde{\kappa}$ as above (see \eqref{eqn:algebraicShintani}). In particular, $\hh(\tilde{\kappa})$ is the ordinary $p$-stabilisation of $h_{\tilde{\kappa}}$ as described in Section \ref{subsec:halfintegral}. Note, however, that the forms $\hh(\tilde\kappa)$ do not define a $\Lambda$-adic form $\hh$ a priori, due to the presence of the periods. With this notation, recall from Theorem \ref{thm:mainCdVP21} that the specialisation of $\pmb{\Theta}$ at a classical point $\tilde{\kappa}\in \widetilde{\mathcal U}_{\hf}^{\cl}$ of weight $k-1$ satisfies  
\[
    \pmb{\Theta}(\tilde{\kappa}) = (-1)^{[k/2]}2^{-k} \cdot \Omega_{\kappa}^- \cdot \hh(\tilde{\kappa}), \qquad 
    \Theta_{\tilde{\kappa}} = (-1)^{[k/2]}2^{-k} \cdot \Omega_{\kappa}^- \cdot h_{\tilde{\kappa}},
\]
where the $p$-adic periods $\Omega_{\kappa}^-$, one for each $\kappa \in \mathcal U_{\hf}^{\cl}$, are as in Theorem \ref{thm:mainGS}. As we did for $\pmb{\Theta}$, here and afterwards we also drop the reference to $r_0$ in the notation, so that $\widetilde{\mathcal U}_{\hf}^{\cl}$ stands for $\widetilde{\mathcal U}_{\hf}^{\cl}(r_0)$. We may fix the choice of $p$-adic periods $\{\Omega_{\kappa}^-\}_\kappa$ to be `centered' with respect to a point $\kappa_0 \in \mathcal U_{\hf}^{\cl}$, meaning that $\Omega_{\kappa_0}^- \neq 0$. Such a choice can be made at a point for which $\hh(\tilde\kappa_0) = \theta_{k,Np,\d}(\hf(\kappa_0)) \neq 0$. This guarantees that $\pmb{\Theta}$ does not vanish in a small enough neighborhood of a point $\tilde{\kappa}_0$ with $\pi(\tilde{\kappa}_0) = \kappa_0$ (cf. Remark \ref{rmk:shintani-vanishing}). For this reason, by shrinking $\mathcal U_{\hf}$ if necessary, from now on we will assume that
\begin{center}
    $\pmb{\Theta}$ does not vanish on $\widetilde{\mathcal U}_{\hf}$.
\end{center}
After this discussion we are ready to prove the following:

\begin{proposition}\label{prop:P}
    Assume $N$ is odd and squarefree. There exists a unique element $\mathcal P(\hf,\pmb{\Theta}) \in \mathrm{Frac}(\widetilde{\mathcal R}_{\hf})$ satisfying the following interpolation property. For every $\tilde{\kappa} \in \widetilde{\mathcal U}_{\mathbf f}^{\cl}$ of weight $k-1$, one has
    \[
    \mathcal P(\hf,\pmb{\Theta}) (\tilde{\kappa}) = \frac{1}{\Omega_{\kappa}^-} \cdot (-1)^{[k/2]}  \cdot 2^{k+1} 
    \cdot
    \frac{\langle f_{\kappa},f_{\kappa}\rangle}{\langle h_{\tilde{\kappa}},h_{\tilde{\kappa}} \rangle \Omega_{f_\kappa}^-},
    \]
    where $\kappa = \pi(\tilde{\kappa})$ and the periods $\Omega_{\kappa}^-$ are as above.
\end{proposition}
\begin{proof}
Let $D$ be an auxiliary fundamental discriminant with $(-1)^{r_0}D > 0$, and $\chi_D(\ell) = w_{\ell}(\hf)$ for each prime $\ell \mid N$, where $w_{\ell}(\hf)$ denotes the eigenvalue of the $\ell$-th Atkin--Lehner involution. Define 
\[
    \mathcal P_D(\hf,\pmb{\Theta}):= 2^{\nu(N)} \frac{ 
    \tilde{\mathcal L}_p(\hf,\chi_{D}\omega^{r_0-1})}{  \mathbf c(|D|)^2} \in \mathrm{Frac}(\widetilde{\mathcal R}_{\hf}),
\]
where $\tilde{\mathcal L}_p(\hf,\chi_{D}\omega^{r_0-1}) \in \widetilde{\mathcal R}_{\hf}$ is the `metaplectic version' of Greenberg--Stevens' $p$-adic $L$-function defined in Proposition \ref{prop:metaplecticGS}. We are going to prove that this element satisfies the claimed property. Let us assume that $\tilde{\kappa} \in \widetilde{\mathcal U}_{\hf}^{\cl}$ is of weight $k-1$, and that $\mathbf f(\kappa)$ is old at $p$, where $\kappa = \pi(\tilde{\kappa})$. In this case, we have 
\[
    \tilde{\mathcal L}_p(\hf,\chi_{D}\omega^{r_0-1})(\tilde{\kappa}) 
    =
    2^{1-k} \cdot \Omega_{\kappa}^- \left(1- \frac{\beta_{f_\kappa} \chi_D(p)}{p^k} \right)^2
    \chi_D(-1)D^{k} \cdot \frac{1}{i^k \mathfrak g(\chi_D)} \cdot \frac{(k-1)!}{\pi^k}\cdot \frac{L(f_{\kappa},D,k)}{\Omega_{f_{\kappa}}^-}.
\]
Since $k$ must be odd, and $r_0$ has the same parity, we have $\chi_D(-1)=\mathrm{sgn}(D)=\mathrm{sgn}(D^k)$, and hence the above interpolation formula reads
\begin{equation}\label{eq:metaplecticLpGS}
    \tilde{\mathcal L}_p(\hf,\chi_{D}\omega^{r_0-1})(\tilde{\kappa}) 
    =
    (-1)^{[k/2]} 2^{1-k} \cdot \Omega_{\kappa}^- \left(1- \frac{\beta_{f_\kappa} \chi_D(p)}{p^k} \right)^2 |D|^{k-1/2}\frac{(k-1)!}{\pi^k}\cdot \frac{L(f_{\kappa},D,k)}{\Omega_{f_{\kappa}}^-}.
\end{equation}
Since $D$ is fundamental, from Proposition \ref{prop:halpha} we have $c_{|D|}(\hh(\tilde\kappa)) =  (1- \beta_{f_\kappa} \chi_D(p)p^{-k}) \cdot c_{|D|}(h_{\tilde\kappa})$, hence 
\begin{align} \label{eqn:ThetaCoeff}
    \mathbf c(|D|)(\tilde{\kappa})^2 
    =
    2^{-2k}\cdot (\Omega_{\kappa}^-)^2 \cdot \left(1- \frac{\beta_{f_\kappa} \chi_D(p)}{p^k} \right)^2\cdot  c_{|D|}(h_{\tilde{\kappa}})^2.
\end{align}
And by Theorem \ref{thm:KohnenFormula} applied with $f = f_\kappa$ and $h = h_{\tilde\kappa}$ (notice that the latter has real coefficients) we find that
\begin{equation*}
    \frac{\langle f_\kappa,f_\kappa \rangle }{\langle h_{\tilde\kappa},h_{\tilde\kappa}\rangle \Omega_{f_{\kappa}}^-}  = 2^{\nu(N)} \frac{ \frac{(k-1)!}{\pi^k}|D|^{k-1/2} \cdot \frac{L(f_{\kappa},D,k)}{\Omega_{f_{\kappa}}^-} }{ c_{|D|}(h_{\tilde\kappa})^2 }.
\end{equation*}
Combining the last formula with the expressions \eqref{eq:metaplecticLpGS} and \eqref{eqn:ThetaCoeff}, we find that $\mathcal P_D(\hf,\pmb{\Theta})$ satisfies the interpolation formula of the statement at points $\tilde{\kappa}$ where $\hf(\kappa)$ is old at $p$. One can easily check that the same proof as above adapts also to the case when $\hf(\kappa)$ is new. Finally, since we clearly see that the interpolation formula is independent on the choice of $D$, so is the element we build, and we may define $\mathcal P(\hf,\pmb{\Theta}) := \mathcal P_D(\hf,\pmb{\Theta})$ for any fundamental discriminant $D$ satisfying the above hypotheses.
\end{proof}

Consider now the $\Lambda$-algebra 
\begin{equation} \label{eqn:OurAlgebra}
    \mathcal R := \widetilde{\mathcal R}_{\hf} \otimes_{\Lambda} \mathcal R_{\hg},
\end{equation}
where the tensor product is taken with respect to the structure morphism, which can be read at the level of weight spaces as 
\[
    \mathcal W_{\mathcal R} := \mathcal W(\mathcal R) = \widetilde{\mathcal W}_{\hf} \times_{\mathcal W} \mathcal W_{\hg}.
\]
At the level of classical points, we may thus identify 
\[
    \mathcal W^{\cl}_{\mathcal R} = \{(\tilde{\kappa},\lambda) \in \widetilde{\mathcal W}^{\cl}_{\hf} \times \mathcal W^{\cl}_{\hg}: \mathrm{wt}(\tilde{\kappa}) = \mathrm{wt}(\lambda)\},
\]
and in a natural way we write $\mathcal U_{\mathcal R}^{\cl}$ for the intersection $\mathcal W_{\mathcal R}^{\cl} \cap \widetilde{\mathcal U}_{\hf}^{\cl} \times \mathcal U_{\hg}^{\cl}$. 

Alternatively, there is a natural isomorphism 
\[
    \mathcal R \simeq \mathcal R_{\hf} \otimes_{\Lambda,\sigma} \mathcal R_{\hg},
\]
where $\sigma : \Lambda \to \Lambda$ is as in Section \ref{sec:padicShintani}. Via this isomorphism, we may instead see $\mathcal W_{\mathcal R}$ as the fibered product 
\[
    \mathcal W_{\mathcal R} \simeq \mathcal W_{\hf} \times_{\mathcal W,\sigma} \mathcal W_{\hg},
\]
which we can read on classical points as: 
\[
    \mathcal W^{\cl}_{\mathcal R} \simeq \{(\kappa,\lambda) \in \mathcal W^{\cl}_{\hf} \times \mathcal W^{\cl}_{\hg}: \mathrm{wt}(\kappa) = 2\mathrm{wt}(\lambda)\}.
\]
Because of the above isomorphism, if $\nu \in \mathcal W^{\cl}_{\mathcal R}$, then we will identify it with either $(\tilde{\kappa},\lambda)$ or $(\kappa,\lambda)$, where $\kappa = \pi(\tilde{\kappa})$. We note that $\mathcal W^{\cl}_{\mathcal R}$ is non-empty if and only if the residue class $\pmod{p-1}$ determined by $\mathcal U_\hf^\cl$ is the same as $2r_0$. We may assume this compatibility condition everywhere from now on.

With this, let $\hSK \in \mathbf S^{(2)}(N,\widetilde{\mathcal R}_{\hf})$ be the $\Lambda$-adic Saito--Kurokawa lift of $\pmb{\Theta}$ described in Proposition \ref{prop:LambdaSK}, set $F_{\tilde\kappa} := \SK_N(h_{\tilde\kappa})$, and let $\hF(\kappa) := \SK_{Np}(\hh(\tilde\kappa))$ be its semi-ordinary $p$-stabilisation as described in equation \eqref{eqn:semiordinary} (note that despite the notation, the latter is not a $\Lambda$-adic family). If $\tilde{\kappa} \in \widetilde{\mathcal U}_{\hf}^{\cl}$ is a classical point of weight $k-1$, and $\kappa = \pi(\tilde{\kappa})$, recall that
\[
    \hSK(\tilde{\kappa}) = (-1)^{[k/2]}2^{-k} \cdot \Omega_{\kappa}^- \cdot \hF(\tilde\kappa).
\]

\begin{proposition}\label{prop:JggF}
    Let $\mathcal J_{\hg\times \hg}$ be the $\Lambda$-adic functional defined in Lemma \ref{lemma:JFg}. Then the element 
    \[
    \mathcal J_{\hg \times \hg}(\eord\varpi(\hSK)) \in \mathrm{Frac}(\mathcal R)
    \]
    satisfies the following interpolation property. If $\nu = (\tilde\kappa,\lambda) \equiv (\kappa,\lambda) \in \mathcal U^{\cl}_{\mathcal R}$ is of weight $k-1$, and $\hf(\kappa)$ and $\hg(\lambda)$ are the ordinary $p$-stabilisations of $f_{\kappa} \in S_{2k}^{\new}(N)$ and $g_{\lambda} \in S_{k+1}^{\new}(N)$, then we have
    \[
        \mathcal J_{\hg \times \hg}(\eord\varpi(\hSK))(\nu) = (-1)^{[k/2]}2^{-k} \cdot \Omega_{\kappa}^- \cdot \frac{\mathcal E^{\circ}(f_{\kappa},\Ad(g_{\lambda}))}{\mathcal E(\Ad(g_{\lambda}))} \cdot \frac{\langle \varpi(F_{\tilde\kappa}), g_{\lambda}\times g_{\lambda} \rangle}{\langle g_{\lambda},g_{\lambda}\rangle^2},
    \]
    where $\mathcal E^{\circ}(f_{\kappa},\Ad(g_{\lambda}))$ and $\mathcal E(\Ad(g_{\lambda}))$ are defined as in equation \eqref{Efg}.
\end{proposition}
\begin{proof}
Directly from Lemma \ref{lemma:JFg}, together with \eqref{eq:interpolation:pullbackSK}, we know that 
\[
    \mathcal J_{\hg\times \hg}(\eord\varpi(\hSK))(\nu) = (-1)^{[k/2]} 2^{-k}  \cdot \Omega_{\kappa}^- \cdot \frac{\langle \eord\varpi(\hF(\tilde{\kappa})), \hg(\lambda)\times \hg(\lambda)\rangle}{\langle \hg(\lambda),\hg(\lambda)\rangle^2},
\]
where we have also used the interpolation property of the $\Lambda$-adic Saito--Kurokawa lift. In order to lighten the notation, fix $\nu = (\kappa,\lambda) \in \mathcal U_{\mathcal R}^{\cl}$, write $f := f_{\kappa}$ and $g := g_{\lambda}$, $g_{\alpha} := \hg(\lambda)$ for the ordinary $p$-stabilisation of $g$, and write also $F_{\alpha}$ for $\hF(\tilde{\kappa})$ and $F:=F_{\tilde{\kappa}}$. Then the ratio of Petersson products on the right hand side of the above identity reads
\[
\frac{\langle \eord\varpi(F_{\alpha}), g_{\alpha} \times g_{\alpha} \rangle}{\langle  g_{\alpha},g_{\alpha}\rangle^2}.
\]
From the description of $\eord$ on the $g\times g$-isotypical subspace of $S_{k+1}(Np) \otimes S_{k+1}(Np)$, the above is just the coefficient of $g_{\alpha} \times g_{\alpha}$ when expressing $F_{\alpha}$ in terms of the basis $g_{\alpha} \times g_{\alpha}$, $g_{\alpha} \times g_{\beta}$, $g_{\beta} \times g_{\alpha}$, $g_{\beta} \times g_{\beta}$. By Corollary \ref{cor:Falpha} we know how to express $\varpi(F_{\alpha})[g]$ with respect to the basis $g \times g$, $g\times Vg$, $Vg \times g$, $Vg \times Vg$, so that we just have to do a change of basis. Indeed, observe that
    \[
        g = \frac{\alpha_g g_{\alpha} - \beta_g g_{\beta}}{\alpha_g-\beta_g}, 
        \qquad 
        Vg = \frac{ g_{\alpha} -  g_{\beta}}{\alpha_g-\beta_g}.
    \]
    In other words, the matrix 
    \[
    \frac{1}{\alpha_g-\beta_g}\begin{pmatrix} \alpha_g & 1 \\ -\beta_g & -1\end{pmatrix}
    \]
    gives the change of basis from $\{g,Vg\}$ to $\{g_{\alpha},g_{\beta}\}$ on $S_{k+1}(Np)$. Taking the tensor product of this matrix with itself, it follows that 
    \[
    \frac{1}{(\alpha_g-\beta_g)^2}\begin{pmatrix} \alpha_g^2 & \alpha_g & \alpha_g & 1 \\ -\alpha_g\beta_g & -\alpha_g & -\beta_g & -1 \\ -\alpha_g\beta_g & -\beta_g & - \alpha_g & -1 \\ \beta_g^2 & \beta_g & \beta_g & 1\end{pmatrix}
    \]
    gives the change on $(S_{k+1}(Np)\otimes S_{k+1}(Np))[g\times g] = S_{k+1}(Np)[g]\otimes S_{k+1}(Np)[g]$ from the basis $\{g\times g, g\times Vg, Vg \times g, Vg\times Vg\}$ to the basis $\{g_{\alpha}\times g_{\alpha}, g_{\alpha}\times g_{\beta}, g_{\beta}\times g_{\alpha}, g_{\beta}\times g_{\beta}\}$. By Corollary \ref{cor:Falpha}, we find 
    \begin{align*}
        \mathrm{e}_\mathrm{ord} (\varpi_g(F_\alpha))
        = 
        \frac{\lambda_{g}\left(1-\frac{\beta_f}{p^k}\right)}{\left( 1-\frac{\beta_g}{\alpha_g} \right)^2 } (A_{\alpha} + B_{\alpha}\alpha_g^{-1} + C_{\alpha}\alpha_g^{-1} + D_{\alpha}\alpha_g^{-2}) g_\alpha \times g_\alpha,
    \end{align*}
    where $\lambda_{g} = \frac{\langle \varpi(F),g\times g\rangle}{\langle g,g\rangle^2}$ and the coefficients $A_{\alpha}$, $B_{\alpha}$, $C_{\alpha}$, and $D_{\alpha}$ are given by the recipe in Corollary \ref{cor:Falpha}. A laborious but elementary computation shows that the coefficient of $g_{\alpha}\times g_{\alpha}$ in this expression is precisely $\mathcal E(f,g) \lambda_{g}$, which proves the formula in the statement (check Corollary \ref{cor:Ordinary Part-appendix} in the Appendix for the detailed calculation).
\end{proof}

We are finally in position to prove the main result of the paper:

\begin{theorem} \label{thm:Main Theorem}
    Suppose that $N$ is odd and squarefree, and let $\hf$ and $\hg$ be two Hida families of tame level $N$ and trivial tame nebentypus character as above. Then there exists a unique element $\mathcal L_p^{\circ}(\hf, \Ad(\hg)) \in \mathrm{Frac}(\mathcal R)$ with the following interpolation property: if $\nu = (\kappa,\lambda) \in \mathcal U_{\hf}^{\cl}$ is of weight $k-1$, and $\hf(\kappa)$ and $\hg(\lambda)$ are the ordinary $p$-stabilisations of $f_{\kappa} \in S_{2k}^{\new}(N)$ and $g_{\lambda} \in S_{k+1}^{\new}(N)$, repectively, then 
    \[
    \mathcal L_p^{\circ}(\hf, \Ad(\hg))(\nu) = \Omega_{\kappa}^- \cdot
    \mathscr C(N,k)^{-1} \cdot \frac{\mathcal E^{\circ}(f_{\kappa},\Ad(g_{\lambda}))^2}{\mathcal E(\Ad(g_{\lambda}))^2}
    \cdot
    \Lambda(f_{\kappa}\otimes \Ad(g_{\lambda}),k)^{\mathrm{alg}},
    \]
    where $\mathcal E(f_{\kappa},\Ad(g_{\lambda}))$ and $\mathcal E(\Ad(g_{\lambda}))$ are as in equation \eqref{Efg} and $\mathscr C(N,k) = (-1)^{[k/2]} 2^{2k} N^{-1} \prod_{q\mid N}(1+q)^2$.
\end{theorem}
\begin{proof}
Uniqueness follows immediately from the interpolation property. In order to prove the existence, recall the element $\mathcal P(\hf,\pmb{\Theta}) \in \mathrm{Frac}(\widetilde{\mathcal R}_{\hf})$ from Proposition \ref{prop:P}. After identifying it with $\mathcal P(\hf,\pmb{\Theta}) \otimes 1 \in \mathrm{Frac}(\mathcal R)$, we define 
\begin{equation} \label{def:The Definition of Lp}
	\mathcal L_p^{\circ}(\hf, \Ad(\hg)) := \mathcal P(\hf,\pmb{\Theta}) \cdot \mathcal J_{\hg\times\hg}(\eord\varpi(\hSK))^2 \in \mathrm{Frac}(\mathcal R).
\end{equation}
By the previous propositions, we find that
\begin{align*}
    \Lp^{\circ}(\hf,\Ad(\hg))(\nu) 
    & = 
    \mathcal P(\hf, \pmb{\Theta}) \cdot \mathcal J_{\hg\times\hg}(\eord\varpi(\hSK))^2 (\nu) 
    = \\
    & = 
    \Omega_{\kappa}^- \cdot 2^{1-k} \cdot \frac{\mathcal E^{\circ}(f_{\kappa},\Ad(g_{\lambda}))^2}{\mathcal E(\Ad(g_{\lambda}))^2} \cdot \frac{\langle f_{\kappa},f_{\kappa}\rangle}{\langle  h_{\tilde{\kappa}},h_{\tilde{\kappa}}\rangle} \cdot \frac{\langle \varpi(F_{\kappa}),g_{\lambda}\times g_{\lambda}\rangle^2}{\langle g_{\lambda},g_{\lambda}\rangle^4 \, \Omega_{f_{\kappa}}^-}.
\end{align*}
Now, from the central value formula in Theorem \ref{thm:PaldVP Introduction} in the introduction and its algebraicity consequence stated in equation \eqref{eqn:PaldVP}, we have that
\[
    \Lambda(f_{\kappa}\otimes \Ad(g_{\lambda}),k)^{\mathrm{alg}} = C(f_{\kappa},g_{\lambda}) \cdot \frac{\langle f_{\kappa},f_{\kappa}\rangle}{\langle h_{\tilde\kappa},h_{\tilde\kappa} \rangle} \cdot \frac{\langle \varpi(F_{\tilde\kappa}),g_{\lambda}\times g_{\lambda}\rangle^2}{\langle g_{\lambda},g_{\lambda}\rangle^4\Omega_{f_\kappa}^-},
\]
where the absence of the absolute value in the last quotient is due to the fact that $F_\kappa$ has real coefficients, since we have chosen $h_{\tilde\kappa} = \theta_{k,\d,Np}^\mathrm{alg}(f_\kappa)$, with real coefficients (see equation \eqref{eqn:algebraicShintani}). Joining the two equations and replacing $C(f_{\kappa},g_{\lambda})$ with its explicit value, the result is achieved.
\end{proof}

One may note that the construction seems to heavily depend on the choice of $\pmb{\Theta}$, while the interpolation formula does not. Indeed, it is easy to see that for any nonzero $\lambda\in \mathrm{Frac}(\widetilde{\mathcal R}_\hf)$ we have
\[
    \mathcal P(\hf, \lambda \pmb{\Theta}) = \lambda^{-2} \mathcal P(\hf, \pmb{\Theta}).
\]
However, the choice of $\lambda\pmb{\Theta}$ also implies the use of $\lambda\hSK$ at the Saito--Kurokawa level, so that
\[ 
    \mathcal J_{\hg\times\hg}(\eord\varpi(\lambda \hSK))^2 = \lambda^2 \mathcal J_{\hg\times\hg}(\eord\varpi(\hSK))^2. 
\]     
And the product does not depend on such a choice. On the same note, a different choice of $\d$ for the construction of the $\Lambda$-adic Shintani lifting would give rise to a lift of the form $\lambda \pmb{\Theta}$ due to multiplicity one results that hold in our case (since $N$ is odd and square-free), so the construction does not depend on $\d$ either. Last, the $p$-adic period $\Omega_\kappa^-$ only depends on the theory of $\Lambda$-adic modular symbol and, as such, is entirely described in terms of $\hf$, so that the final construction solely depends on the pair $(\hf,\hg)$.

\section{A factorisation of \texorpdfstring{$p$}{p}-adic \texorpdfstring{$L$}{L}-functions}\label{sec:factorisation}

Let $N\geq 1$ be an odd, squarefree integer, and let \[
f \in S_k^{\new}(N), \quad g \in S_{\ell}^{\new}(N), \quad h \in S_m^{\new}(N)
\]
be three normalised newforms of weights $k, \ell, m$, respectively, and of common level $\Gamma_0(N)$ for simplicity. Let $L(f \otimes g \otimes h, s)$ denote the Garret--Rankin $L$-series associated with the tensor product $V_f \otimes V_g \otimes V_h$ of (compatible systems of) Galois representations attached to $f$, $g$, and $h$. This can be written as an Euler product 
\[
	L(f \otimes g \otimes h, s) = \prod_q L_{(q)}(f \otimes g \otimes h, q^{-s})^{-1},
\]
for $\mathrm{Re}(s) \gg 0$, where for a prime $q \nmid N$ one has 
\begin{align*}
    L_{(q)}(f \otimes g \otimes h, T) = & (1-\alpha_q(f)\alpha_q(g)\alpha_q(h) T)(1-\alpha_q(f)\alpha_q(g)\beta_q(h) T)(1-\alpha_q(f)\beta_q(g)\alpha_q(h) T)(1-\alpha_q(f)\beta_q(g)\beta_q(h) T) \\
    & (1-\beta_q(f)\alpha_q(g)\alpha_q(h) T)(1-\beta_q(f)\alpha_q(g)\beta_q(h) T)(1-\beta_q(f)\beta_q(g)\alpha_q(h) T)(1-\beta_q(f)\beta_q(g)\beta_q(h) T).
\end{align*}
This complex $L$-series can be completed with an archimedean factor, whose definition depends on the weights of $f$, $g$, and $h$. The triple of weights $(k,\ell,m)$ is said to be {\em balanced} if $k+\ell+m> 2 \mathrm{max}\{k,\ell,m\}$. If $(k,\ell,m)$ is not balanced, reordering the newforms we may assume that the `dominant' weight is the one of $f$, so that $k \geq \ell + m$. With this, the completed $L$-series $\Lambda(f\otimes g \otimes h, s) := L_{\infty}(f \otimes g \otimes h,s)L(f\otimes g\otimes h,s)$ is defined by setting 
\[
L_{\infty}(f \otimes g \otimes h,s) := \begin{cases}
\Gamma_{\C}(s)\Gamma_{\C}(s+k-2c)\Gamma_{\C}(s+1-\ell)\Gamma_{\C}(s+1-m) & \text{if } k \geq \ell+m, \\
\Gamma_{\C}(s)\Gamma_{\C}(s+1-k)\Gamma_{\C}(s+1-\ell)\Gamma_{\C}(s+1-m) & \text{if } (k,\ell,m) \text{ is balanced,}
\end{cases}
\]
where $c = (k+\ell+m-2)/2$. It admits analytic continuation to the whole complex plane, and satisfies a functional equation relating the values $\Lambda(f\otimes g \otimes h,s)$ and $\Lambda(f\otimes g \otimes h, 2c-s)$, so that $s = c$ is the center of symmetry for the functional equation. The global sign $\varepsilon(f\otimes g \otimes h) = \pm 1$ appearing in this functional equation determines the parity of the order of vanishing for $\Lambda(f\otimes g \otimes h,s)$ at $s = c$, and can be expressed as a product of local signs 
\[
\varepsilon(f\otimes g\otimes h) = \prod_v \varepsilon_v(f\otimes g\otimes h), \qquad \varepsilon_v(f\otimes g\otimes h) \in \{\pm 1\},
\]
varying over all the rational places. It is known that $\varepsilon_v(f\otimes g\otimes h) = +1$ for all $v \nmid N\infty$, and  
\[
\varepsilon_{\infty}(f\otimes g\otimes h) = \begin{cases} 
+1 & \text{if } (k,\ell,m) \text{ is not balanced,}\\
-1 & \text{if } (k,\ell,m) \text{ is balanced.}
\end{cases}
\]

Assume from now on that the triple $(k,\ell,m)$ is balanced, and that 
\[
\Sigma^- := \{q \mid N: \varepsilon_q(f\otimes g \otimes h) = -1\}
\]
has {\em odd} cardinality. Then, the global sign $\varepsilon(f\otimes g \otimes h)$ is $+1$, and hence $\Lambda(f\otimes g \otimes h,c)$ is not forced to vanish by sign reasons. In this case, the value
\[
    \Lambda(f\otimes g \otimes h,c)^{\mathrm{alg}} := \frac{\Lambda(f\otimes g \otimes h,c)}{\langle f,f \rangle\langle g,g \rangle\langle h,h \rangle} \in \barQ
\]
is algebraic. When letting $f$, $g$, and $h$, vary in Hida families $\hf$, $\hg$, and $\hh$, these algebraic central values can be interpolated by a three-variable $p$-adic $L$-function, as constructed in \cite{Hsieh}. 

Indeed, let $\hf$, $\hg$, and $\hh$ be three Hida families of ordinary $p$-stabilised newforms of tame level $N$ as in Definition \ref{def:Hidafamily}. Let $\mathcal R_{\hf,\hg,\hh} := \mathcal R_{\hf}\otimes\mathcal R_{\hg}\otimes \mathcal R_{\hh}$, and 
\[
    \mathcal U_{\hf,\hg,\hh} = \mathcal U_{\hf} \times \mathcal U_{\hg} \times \mathcal U_{\hh} \subset \mathcal W_{\hf} \times \mathcal W_{\hg} \times \mathcal W_{\hh}
\]
be the associated weight space. Let $x=(\kappa,\lambda,\mu) \in \mathcal U_{\hf,\hg,\hh}^{\cl} := \mathcal U_{\hf}^{\cl}\times \mathcal U_{\hg}^{\cl}\times \mathcal U_{\hh}^{\cl}$ be a classical point, and let $\Sigma^-(x)$ denote the set of primes $q \mid N$ such that the local sign $\varepsilon_q(f_{\kappa}\otimes g_{\lambda}\otimes h_{\mu})$ in the functional equation for the Garret--Rankin $L$-series $L(f_{\kappa}\otimes g_{\lambda}\otimes h_{\mu},s)$ is $-1$. It is well-known that $\Sigma^-(x)$ actually does not depend on the choice of classical point $x$, and so we may write just $\Sigma^-$ in analogy with the above classical discussion. Assume that
\begin{equation}\label{signassumption}
    \Sigma^- \quad \text{has odd cardinality,}
\end{equation}
so that the Garret--Rankin $L$-series $L(f_{\kappa},g_{\lambda},h_{\mu},s)$ has global sign $+1$ for every triple of classical weights $(\kappa,\lambda,\mu)$ in the {\em balanced region}\footnote{Recall our convention that the weight of a classical specialisation of a Hida family differs by two with the weight of the corresponding classical point in the weight space.}
\[
\mathcal U_{\hf,\hg,\hh}^{\mathrm{bal}} := \left\lbrace (\kappa,\lambda,\mu) \in \mathcal U_{\hf,\hg,\hh}^{\cl}: \mathrm{wt}(\kappa)+\mathrm{wt}(\lambda)+\mathrm{wt}(\mu) > 2\mathrm{max}\{\mathrm{wt}(\kappa),\mathrm{wt}(\lambda),\mathrm{wt}(\mu)\}-2\right\rbrace.
\]
Assume also that $\hf$, $\hg$, and $\hh$ fulfill Hypothesis (CR, $\Sigma^-$) in \cite{Hsieh}. Then, by Theorem B in op. cit., there exists a unique $\mathcal L_p^{\mathrm{bal}}(\hf,\hg,\hh) \in \mathrm{Frac}(\mathcal R_{\hf,\hg,\hh})$, which we will regard as a function $\mathcal L_p^{\mathrm{bal}}(\hf,\hg,\hh): \mathcal U_{\hf,\hg,\hh} \, \to \, \C_p$, satisfying the interpolation property that 
\begin{equation}\label{eqn:Hsiehinterpolation}
    (\mathcal L_p^{\mathrm{bal}}(\hf,\hg,\hh)(\kappa,\lambda,\mu))^2 = 2^{-(k+\ell+m+3)} \frac{\mathcal E(f_{\kappa},g_{\lambda},h_{\mu})^2}{\mathcal E(\Ad(f_{\kappa}))\mathcal E(\Ad(g_{\lambda})) \mathcal E(\Ad(h_{\mu}))} \cdot \Lambda(f_{\kappa}\otimes g_{\lambda}\otimes h_{\mu},c)^\mathrm{alg},
\end{equation}
for every triple $(\kappa,\lambda,\mu) \in \mathcal U_{\hf,\hg,\hh}^{\mathrm{bal}}$ with $\mathrm{wt}(\kappa) = k-2$, $\mathrm{wt}(\lambda)=\ell-2$, $\mathrm{wt}(\mu)=m-2$, and such that $\hf(\kappa)$, $\hg(\lambda)$, and $\hh(\mu)$ are the ordinary $p$-stabilisations of $f_{\kappa}\in S_k^{\new}(N)$, $g_{\lambda} \in S_{\ell}^{\new}(N)$, and $h_{\mu} \in S_m^{\new}(N)$, respectively. Here, $\mathcal E(\Ad(-))$ is defined as in \eqref{Efg}, and writing $f = f_{\kappa}$, $g = g_{\lambda}$, and $h = h_{\mu}$ for simplicity, we set
\[
\mathcal E(f,g,h) := \left(1-\frac{\alpha_f\beta_g\beta_h}{p^c}\right)\left(1-\frac{\beta_f\alpha_g\beta_h}{p^c}\right)\left(1-\frac{\beta_f\beta_g\alpha_h}{p^c}\right) \left(1-\frac{\beta_f\beta_g\beta_h}{p^c}\right).
\]
We warn the reader that our normalisation for $\mathcal L_p^{\mathrm{bal}}(\hf,\hg,\hh)$ differs slightly from the one in \cite{Hsieh}; in particular, we do not claim that $\mathcal L_p^{\mathrm{bal}}(\hf,\hg,\hh)$ belongs to $\mathcal R_{\hf,\hg,\hh}$, but only to its fraction field.

Now, let us change a bit the notation to one which is reminiscent from previous sections. Continue to assume that $N\geq 1$ is odd and squarefree, let $\ell \geq k \geq 1$ be odd integers, and let $f \in S_{2k}^{\new}(N)$, $g \in S_{\ell+1}^{\new}(N)$ be normalised newforms of level $N$ and weights $2k$ and $\ell+1$, respectively. In this case, the decomposition of representations 
\[
V_g \otimes V_g \simeq \det(V_g) \otimes (\mathbf 1 \oplus \Ad(V_g))
\]
yields, by Artin formalism, a factorisation of complex $L$-functions \begin{equation}\label{complexfact-incomplete}
L(f\otimes g \otimes g,s) = L(f,s-\ell)L(f\otimes \Ad(g),s-\ell).
\end{equation}
After completing each of the three complex $L$-functions in \eqref{complexfact-incomplete} with the corresponding archimedean Euler factors, one gets a factorisation of completed $L$-functions
\begin{equation}\label{complexfact-complete}
\Lambda(f\otimes g \otimes g,s) = \Lambda(f,s-\ell)\Lambda(f\otimes \Ad(g),s-\ell).
\end{equation}
In addition, recalling that we choose Shimura's complex periods $\Omega_f^{\pm} \in \C^{\times}$ attached to $f$ with the extra property that $\Omega_f^+\Omega_f^- = \langle f, f \rangle$, the above identity also yields a factorisation of algebraic central values 
\begin{equation}\label{eqn:algebraicfactorization}
\Lambda(f\otimes g \otimes g,k+\ell)^{\mathrm{alg}} = \Lambda(f,k)^{\mathrm{alg}}\Lambda(f\otimes \Ad(g),k)^{\mathrm{alg}},
\end{equation}
where 
\[
\Lambda(f\otimes \Ad(g),k)^{\mathrm{alg}} := \frac{\Lambda(f\otimes \Ad(g),k)}{\langle g,g\rangle^2\Omega_f^-} \in \Q(f,g)
\]
is algebraic as commented in \eqref{eqn:PaldVP}, and 
\[
\Lambda(f\otimes g \otimes g,k+\ell)^{\mathrm{alg}} := \frac{\Lambda(f\otimes g \otimes g,k+\ell)}{\langle f,f\rangle\langle g,g\rangle^2} \in \Q(f,g), \quad  \Lambda(f,k)^{\mathrm{alg}} := \frac{\Lambda(f,k)^{\mathrm{alg}}}{\Omega_f^+} \in \Q(f)
\]
are also algebraic. Note that the choice of weights for $f$ and $g$ makes that the triple $(2k,\ell+1,\ell+1)$ is {\em balanced}, according to the above introduced terminology. Assuming the same sign hypothesis as above, one expects the above complex factorisations to be mirrored on the $p$-adic side by a (two-variable) factorisation of the corresponding $p$-adic $L$-functions. Falling a bit short in this wish, we prove in this section the expected factorisation when restricted to a suitable line, i.e. we prove a one-variable  factorisation of $p$-adic $L$-functions suggested by the above complex discussion.

To be precise, suppose that $\hf$ and $\hg$ are Hida families of ordinary $p$-stabilised newforms of tame level $N$. Using the same notation as in previous sections, consider the $\Lambda$-algebra 
\[
\mathcal R := \widetilde{\mathcal R}_{\hf} \otimes_{\Lambda} \mathcal R_{\hg} \simeq \mathcal R_{\hf} \otimes_{\Lambda,\sigma} \mathcal R_{\hg},
\]
and its associated weight space 
\[
\mathcal W_{\mathcal R} = \widetilde{\mathcal W}_{\hf} \times_{\mathcal W} \mathcal W_{\hg} \simeq \mathcal W_{\hf} \times_{\mathcal W,\sigma} \mathcal W_{\hg}.
\]
Recall that, on classical weights, we have 
\[
\mathcal W_{\mathcal R} \simeq \left\lbrace (\kappa,\lambda) \in \mathcal W_{\hf}^{\cl} \times \mathcal W_{\hg}^{\cl}: \mathrm{wt}(\kappa) = 2\mathrm{wt}(\lambda)\right\rbrace.
\]
In particular, observe that there is a natural embedding
\[
\iota: \mathcal W_{\mathcal R} \, \hookrightarrow \mathcal W_{\hf,\hg,\hg} = \mathcal W_{\hf} \times \mathcal W_{\hg} \times \mathcal W_{\hg}, \quad (\kappa,\lambda) \, \mapsto \, (\kappa,\lambda,\lambda),
\]
through which we can identify $\mathcal U_{\mathcal R}^{\cl}$ with the classical triples $(\kappa,\lambda,\lambda) \in \mathcal U^{\cl}_{\hf,\hg,\hg}$ with $\mathrm{wt}(\kappa) = 2\mathrm{wt}(\lambda)$. The image of $\mathcal U_{\mathcal R}^{\cl}$ lies actually in the {\em balanced region} $\mathcal U_{\hf,\hg,\hg}^{\mathrm{bal}} \subset \mathcal U_{\hf,\hg,\hg}^{\cl}$ defined as above. With this, assuming hypotheses \eqref{signassumption} and (CR, $\Sigma^-$) as above, let 
\[
\mathcal L_p^{\mathrm{bal},\circ}(\hf,\hg,\hg): \mathcal U_{\mathcal R} \, \stackrel{\iota}{\hookrightarrow} \, \mathcal U_{\hf,\hg,\hg} \, \to \, \C_p
\]
be the restriction of the $p$-adic $L$-function $\mathcal L_p^{\mathrm{bal}}(\hf,\hg,\hg)$ explained above (with $\hh = \hg$) to $\mathcal U_{\mathcal R}$ through the embedding $\iota$.

On the other hand, consider the pullback 
\[
\mathcal L_p^{\circ}(\hf,\omega^{r_0-1}): \mathcal U_{\mathcal R} \, \longrightarrow \, \widetilde{\mathcal U}_{\hf} \, \longrightarrow \, \C_p
\]
of the one-variable Greenberg--Stevens $p$-adic $L$-function $\widetilde{\mathcal L}_p(\hf,\omega^{r_0-1})$ normalised as in Proposition \ref{prop:metaplecticGS}, via the natural map $\mathcal U_{\mathcal R} \to \widetilde{\mathcal U}_{\hf}$, $\nu \mapsto \tilde{\kappa}$.

\begin{theorem} \label{thm:MainFactorization}
There exists a factorisation of $p$-adic $L$-functions
\[
    \mathcal L_p^{\mathrm{bal},\circ}(\hf,\hg,\hg)^2 =  \eta \cdot \mathfrak C \cdot \left( 1- \frac{a_p^\circ(\hf)}{a_p^\circ(\hg)^2}  \right)^2 \cdot \mathcal L_p^{\circ}(\hf, \Ad(\hg)) \cdot \mathcal L_p^\circ(\hf,\omega^{r_0-1}),
\]
where $\eta,\mathfrak C\in \mathrm{Frac}(\mathcal R)$ are non trivial functions such that 
\begin{align*}
    \eta(\nu) =  \left(\mathcal E(\Ad(f_{\kappa}))  \Omega_\kappa^+\Omega_\kappa^-\right)^{-1},
    \qquad 
    \mathfrak C(\nu) = -i 2^{5-2k} N^{-1}\prod_{q\mid N}(1+q)^2,
\end{align*}
for all $\nu = (\kappa,\lambda) \in \mathcal U_{\mathcal R}^\cl$, and where $a_p^{\circ}(\hf)$, $a_p^{\circ}(\hg)$ are naturally defined by setting $a_p^{\circ}(\hf)(\nu) = a_p(\hf(\kappa))$ and $a_p^{\circ}(\hg)(\nu) = a_p(\hg(\lambda))$.
\end{theorem}
\begin{proof}
We will check the claimed factorisation at an arbitrary classical point $\nu = (\kappa,\lambda) \in \mathcal U_{\mathcal R}^{\cl}$ for which $\hf(\kappa)$ and $\hg(\lambda)$ are the ordinary $p$-stabilisations of newforms $f_{\kappa} \in S_{2k}^{\new}(N)$ and $g_{\lambda} \in S_{k+1}^{\new}(N)$, respectively. So let us fix such a point for the rest of the proof.

One easily checks that the interpolation property in \eqref{eqn:Hsiehinterpolation} implies that 
\begin{equation}\label{interpolation:LPbal0}
    (\mathcal L_p^{\mathrm{bal},\circ}(\hf,\hg,\hg)(\nu))^2 = 2^{-(4k+5)} \frac{\mathcal E(f_{\kappa},g_{\lambda},g_{\lambda})^2}{\mathcal E(\Ad(f_{\kappa}))\mathcal E(\Ad(g_{\lambda}))^2} \cdot \Lambda(f_{\kappa}\otimes g_{\lambda}\otimes g_{\lambda},2k)^{\mathrm{alg}},
\end{equation}
whereas the interpolation property for the $p$-adic $L$-function $\mathcal L_p^{\circ}(\hf,\omega^{r_0-1})$ tells us that 
\begin{equation}\label{Lpf0-interpolation}
   \mathcal L_p^{\circ}(\hf,\omega^{r_0-1})(\nu) =  \Omega_{\kappa}^+ \mathcal E(f_\kappa, \omega^{r_0-1},k) i^{-k}\cdot \Lambda(f_\kappa,k)^\mathrm{alg}.
\end{equation}
Using $\alpha_{g_\lambda}\beta_{g_\lambda} = p^k$ it is easily seen that
\begin{align*}
    \mathcal E(f_{\kappa},g_{\lambda},g_{\lambda})^2 = 
    \left(1-\frac{\alpha_{f_{\kappa}}}{\alpha_{g_{\lambda}}^2} \right)^2\mathcal E^{\circ}(f_\kappa, \Ad(g_\lambda))^2 \cdot \mathcal E(f_\kappa, \omega^{r_0-1},k).
\end{align*}
Thus combining equation \eqref{eqn:algebraicfactorization} with equations \eqref{interpolation:LPbal0}, \eqref{Lpf0-interpolation}, and the interpolation formula in Theorem \ref{thm:Main Theorem} we find
\[
    (\mathcal L_p^{\mathrm{bal},\circ}(\hf,\hg,\hg)(\nu))^2 
    =  
    (\Omega_{\kappa}^+\Omega_\kappa^-\mathcal E(\Ad(f_{\kappa}))^{-1}\frac{i^{k} \mathscr C(N,k)}{2^{4k+5}} \left(1-\frac{\alpha_{f_\kappa}}{\alpha_{g_\lambda}^2}\right)^2
    \mathcal L_p^{\circ}(\hf, \Ad(\hg))(\nu) \cdot \mathcal L_p^\circ(\hf,\omega^{r_0-1})(\nu).
\]
Using the expression of Theorem \ref{thm:mainthm-intro} for $\mathscr C(N,k)$, and the fact that $i^k(-1)^{[k/2]}=-i$ since $k$ is odd, we obtain
\[
    (\mathcal L_p^{\mathrm{bal},\circ}(\hf,\hg,\hg)(\nu))^2 =  \eta(\nu) \cdot \mathfrak C(\nu) \cdot \left( 1- \frac{\alpha_{f_\kappa}}{\alpha_{g_\lambda}^2} \right)^2 \cdot \mathcal L_p^{\circ}(\hf, \Ad(\hg))(\nu) \cdot \mathcal L_p^\circ(\hf,\omega^{r_0-1})(\nu).
\]
Since all the functions appearing are $\Lambda$-adic, possibly except $\eta$, it follows that there exists a unique $\Lambda$-adic function $\eta\in \mathcal R$ extending $\eta(\nu)$ in a small enough neighborhood.
\end{proof}

As our notation suggests, the factorisation in this theorem must be seen as the one-variable shadow of more general factorisations of $p$-adic $L$-functions in two and three variables. Indeed, the balanced $p$-adic triple product $L$-function introduced above has been recently extended in \cite{HsiehYamana} to a four-variable $p$-adic $L$-function incorporating a cyclotomic variable, that we still denote $\mathcal L_p^{\mathrm{bal}}(\hf,\hg,\hg)$ by abuse of notation. Write 
\[
    \mathcal L_p^{\mathrm{bal},\star}(\hf,\hg,\hg): \mathcal U_{\hf,\hg} \times \mathcal W \, \longrightarrow \, \C_p
\]
for its restriction via the natural embedding 
\[
\mathcal U_{\hf,\hg} \times \mathcal W \, \longrightarrow \, \mathcal U_{\hf,\hg,\hg} \times \mathcal W, \quad (\kappa,\lambda,\sigma) \, \longmapsto \, (\kappa,\lambda,\lambda,\sigma).
\]
And similarly, write 
\[
\mathcal L_p^{\star}(\hf,\omega^{r_0-1}): \mathcal U_{\hf,\hg} \times \mathcal W \, \longrightarrow \, \C_p
\]
for the pullback of $\mathcal L_p(\hf,\omega^{r_0-1})$ via the natural morphism 
\[
\mathcal U_{\hf,\hg} \times \mathcal W \, \longrightarrow \, \mathcal U_{\hf} \times \mathcal W.
\]
Then, we can envisage a {\em three-variable} $p$-adic $L$-function
\[
\mathcal L_p^{\star}(\hf, \Ad(\hg)) := \frac{\mathcal L_p^{\mathrm{bal},\star}(\hf,\hg,\hg)}{\eta \cdot \mathfrak C \cdot \mathcal L_p^{\star}(\hf)},
\]
always under the relevant sign assumption. The following diagram may help to understand the sources of the different $p$-adic $L$-functions involved in the above digression, with the different embeddings used for the restrictions.
\begin{equation}
    \xymatrix{
        \mathcal U_{\mathcal R} \ar[r]^{\nu \mapsto (\kappa,\lambda)} &\mathcal U_\hf\times \mathcal U_\hg \ar[d]^{(\kappa,\lambda)\mapsto (\kappa,\lambda,\lambda)} \ar[r]   &\mathcal U_\hf\times \mathcal U_\hg \times \mathcal W \ar[d]  & \mathcal L_p^{\star}(\hf, \Ad(\hg)) \, ? 
        \\
        &\mathcal U_\hf\times \mathcal U_\hg \times \mathcal U_\hg \ar[d]\ar[r]   &\mathcal U_\hf\times \mathcal U_\hg \times \mathcal U_\hg \times \mathcal W \ar[d]  & \mathcal L_p^{\mathrm{bal}}(\hf,\hg,\hg)
        \\
        &\mathcal U_\hf \ar[r]^{ \kappa \mapsto (\kappa,\mathrm{wt}(\kappa)/2+1) } &\mathcal U_\hf \times \mathcal W & \mathcal L_p(\hf,\omega^{r_0-1})
    }
\end{equation}
On the right hand side, each of the $p$-adic $L$-functions listed is written next to its natural domain. On the left hand side, the one-variable $p$-adic $L$-functions $\mathcal L_p^{\circ}(\hf,\Ad(\hg))$, $\mathcal L_p^{\mathrm{bal},\circ}(\hf,\hg,\hg)$, and $\mathcal L_p^{\circ}(\hf,\omega^{r_0-1})$ are defined on $\mathcal U_{\mathcal R}$, and $\mathcal L_p^{\mathrm{bal},\circ}(\hf,\hg,\hg)$ and $\mathcal L_p^{\circ}(\hf)$ are restrictions of $\mathcal L_p^{\mathrm{bal}}(\hf,\hg,\hg)$ and $\mathcal L_p(\hf,\omega^{r_0-1})$ through the obvious maps. The (expected) two-variable $p$-adic $L$-function alluded to in Remark \ref{rmk:twovariable-intro} would be the restriction of $\mathcal L_p^{\star}(\hf,\Ad(\hg))$ under the obvious map.

It would be interesting to put the above picture into a broader framework, by embedding it within the line of work started by Schmidt \cite{Schmidt} and culminating with Januszewski's contributions in \cite{Jan15, Jan}. Without entering into a detailed study, let us just mention here that our $p$-adic $L$-function and the one-variable (cyclotomic) Rankin--Selberg $p$-adic $L$-function by Schmidt (further generalized in \cite{Jan15}) are defined in different domains. However, the recent work of Januszewski in \cite{Jan} provides a $p$-adic $L$-function (varying both on the weight variable and the cyclotomic one) that should specialize to both, in the corresponding domains. This idea is represented in the diagram
\begin{equation}
	\xymatrix{
		\mathcal L_p^{\circ}(\hf, \Ad(\hg)) &\mathcal U_{\mathcal R} \ar[dr]_{\nu \mapsto (\kappa,\lambda)}  & & \mathcal W\ar[ld]^{\sigma\mapsto (\kappa_0,\lambda_0,\sigma)} & \mathcal L_p^{\mathrm{Schmidt}}(f_{\kappa_0}\otimes\Ad(g_{\lambda_0}))
		\\
		& &\mathcal U_\hf\times \mathcal U_\hg \times \mathcal W  & & \mathcal L_p^{\star}(\hf, \Ad(\hg)) \, ? 
	}
\end{equation}
For arithmetic purposes, it would be desirable to achieve a direct construction of $\mathcal L_p^{\star}(\hf, \Ad(\hg))$, which relies on explicit formulas and that is independent of any sign assumption, fitting in this big picture.

\appendix 

\section{Computations from Section \ref{sec:Pullback}}\label{sec:appendix}

In this appendix we complete the proofs of Theorem \ref{thm:PullbackUF} and Corollary \ref{cor:Falpha} in Section \ref{sec:Pullback}. Recall the notation in that passage: $f \in S_{2k}^{\new}(N)$ is a normalised newform, and $F \in S_{k+1}^{(2)}(N)$ is the Saito--Kurokawa lift of a (non-zero) half-integral weight cusp form $h \in S_{k+1/2}^{+,new}(N)$ in Shimura--Shintani correspondence with $f$. To lighten the notation, in what follows we will abbreviate $a_p = a_p(f)$. We start with the following proposition, which fills the omitted computations in the proof of Theorem \ref{thm:PullbackUF}.

\begin{proposition}\label{prop:PullbackUF-appendix}
    Let $\phi \in S_{k+1}(N)$ be a normalised eigenform, and suppose that $\varpi_{\phi}(F) = \lambda_{\phi} \phi\times \phi$. Then 
    \[
        \varpi_{\phi}(UF) = \lambda_{\phi}\left(A\phi \times \phi + B\phi\times V\phi + CV\phi\times\phi + D V\phi\times V\phi\right),
    \]
    where 
    \[
        A = \frac{(p^{k-1}(p-1) + a_p) a_p(\phi)^2 - p^k(p+1)(p^{k-1}(p+1)+a_p)}{a_p(\phi)^2 - p^{k-1}(p+1)^2},
    \]
    and $B$, $C$, $D$ are given by the formulae
    \[
        B = C = \frac{p a_p(\phi)}{p+1}(p^k-p^{k-1}+a_p-A), \quad D = -p^{k+1}(p^k+a_p-A).
    \]
\end{proposition}
\begin{proof}
 As indicated in the proof of Theorem \ref{thm:PullbackUF}, one has to solve the system of linear equations arising from the identity \eqref{eq:Xi[phi]}. One easily checks that $C = B$, and then the system to solve becomes 
 \[
    \begin{cases}
    (p^k-p^{k-1}+a_p)a_p(\phi)^2-p^k(p^{k-1}+a_p) = A(a_p(\phi)^2-p^{k-1})+2Ba_p(\phi) + D, \\
    a_p(\phi)p(p^k-p^{k-1}+a_p) = Aa_p(\phi)p + B(p+1), \\
    p^{k+1}(p^k+a_p) = Ap^{k+1}-D.
    \end{cases}
 \]
 From the last equation, we get 
 \[
    D = -p^{k+1}\left(p^k + a_p - A \right),
 \]
 and from the second one, 
 \[
    B = \frac{pa_p(\phi)}{p+1}\left(p^k-p^{k-1}+a_p-A\right).
 \]
 Plugging these expressions into the first equation of the above system yields
 \[
    \left(p^{k-1}(p-1)+a_p - \frac{2p^{k+1}-2p^k+2pa_p}{p+1}\right)a_p(\phi)^2 -p^k(1-p)(p^{k-1}(p+1)+a_p) = 
 \]
 \[
    = A\left(a_p(\phi)^2 + p^{k-1}(p^2-1) - \frac{2a_p(\phi)^2p}{p+1}\right).
 \]
 Multiplying by $p+1$ and dividing by $1-p$ both sides of the equality, one eventually finds
 \[
    A = \frac{(p^{k-1}(p-1) + a_p) a_p(\phi)^2 - p^k(p+1)(p^{k-1}(p+1)+a_p)}{a_p(\phi)^2 - p^{k-1}(p+1)^2}.
 \]
\end{proof}

Once we have described $\varpi(UF)$, the following corollary completes the proof of Corollary \ref{cor:Falpha}. As usual, write $\alpha_f$, $\beta_f$ for the roots of the $p$-th Hecke polynomial of $f$, and let $F_{\alpha} = \SK_{Np}(h_{\alpha}) \in S_{k+1/2}^+(Np)$ be the Saito--Kurokawa lift of the $p$-stabilisation of $h$ on which $U_{p^2}$ acts with eigenvalue $\alpha_f$.

\begin{corollary}\label{cor:Falpha-appendix}
Let $\phi \in S_{k+1}(N)$ be a normalised eigenform, and $\alpha_{\phi}$, $\beta_{\phi}$ be the roots of its $p$-th Hecke polynomial. Then we have 
\[
    \varpi_{\phi}(F_\alpha) = \lambda_\phi \left(1-\frac{\beta_f}{p^k}\right)\left( A_{\alpha} \phi\times\phi + B_{\alpha}\phi\times V\phi + C_{\alpha}V\phi\times \phi + D_{\alpha}V\phi\times V\phi\right),
\]
where the coefficients $A_{\alpha}$, $B_{\alpha}$, $C_{\alpha}$, and $D_{\alpha}$ are given by the formulae
\[
A_{\alpha} = 1 -\frac{(p+1)\left(1-\frac{\beta_f}{p^{k-1}}\right)}{\left(p-\frac{\alpha_{\phi}}{\beta_{\phi}}\right)\left(1-\frac{\beta_{\phi}}{p\alpha_{\phi}}\right)},
\]
\[
B_{\alpha} = C_{\alpha} = \frac{pa_p(\phi)}{p+1}(1-A_{\alpha}), \quad D_{\alpha} = p^{k+1}(A_{\alpha}-1)-p\beta_f.
\]
\end{corollary}
\begin{proof}
From equation \eqref{eqn:semiordinary}, we know that 
\[
\alpha_f F_{\alpha} = (U-\beta_f)(1-p^kV)F = UF - (p^k+\beta_f)F + \beta_f p^k VF.
\]
Therefore, 
\[
\alpha_f \varpi_{\phi}(F_{\alpha}) = \varpi_{\phi}(UF) - (p^k+\beta_f)\varpi_{\phi}(F) + \beta_f p^k V \times V \varpi_{\phi}(F).
\]
Writing $\varpi_{\phi}(F) = \lambda_{\phi}\phi\times \phi$, it follows from Proposition \ref{prop:PullbackUF-appendix} that $\varpi_{\phi}(F_{\alpha})$ vanishes if $\lambda_{\phi} = 0$. Thus we may write 
\[
\varpi_{\phi}(F_\alpha) = \lambda_{\phi}\left(1-\frac{\beta_f}{p^k}\right)\left( A_{\alpha} \phi\times\phi + B_{\alpha}\phi\times V\phi + C_{\alpha}V\phi\times \phi + D_{\alpha}V\phi\times V\phi\right)
\]
for some coefficients $A_{\alpha}$, $B_{\alpha}$, $C_{\alpha}$, $D_{\alpha}$ to determine. Letting $A$, $B$, $C$, and $D$ be as in Proposition \ref{prop:PullbackUF-appendix}, and using that $\alpha_f\beta_f = p^{2k-1}$, we deduce that 
\[
(\alpha_f-p^{k-1})A_{\alpha} = A - (p^k+\beta_f), \quad (\alpha_f-p^{k-1})B_{\alpha} = B, \quad (\alpha_f-p^{k-1})C_{\alpha} = C, \quad (\alpha_f-p^{k-1})D_{\alpha} = D + \beta_fp^k.
\]
And now we only have to use the values for $A$, $B$, $C$, $D$ obained in Proposition \ref{prop:PullbackUF-appendix}. Indeed, noticing that 
\[
(p^k+\beta_f)(a_p(\phi)^2-p^{k-1}(p+1)^2) = p^ka_p(\phi)^2 + \beta_f a_p(\phi)^2 - p^{2k-1}(p+1)^2 - \frac{\beta_f}{p} p^k(p+1)^2,
\]
from Proposition \ref{prop:PullbackUF-appendix} we have
\[
A - (p^k+\beta_f) = \frac{(\alpha_f-p^{k-1})a_p(\phi)^2 - (\alpha_f-p^{k-1})(p^k+\beta_f)(p+1)}{a_p(\phi)^2-p^{k-1}(p+1)^2},
\]
and hence we deduce that
\[
A_{\alpha} =  \frac{a_p(\phi)^2 - (p^k+\beta_f)(p+1)}{a_p(\phi)^2-p^{k-1}(p+1)^2}.
\]
Letting $\alpha_{\phi}$ and $\beta_{\phi}$ denote the roots of the $p$-th Hecke polynomial for $\phi$, so that $\alpha_{\phi}+\beta_{\phi} = a_p(\phi)$ and $\alpha_{\phi} \beta_{\phi} = p^k$, one checks that the above is equivalent to the expression given in the statement.

Once we have determined $A_{\alpha}$, the rest of coefficients follow easily. First, observe that $C=B$ implies $C_{\alpha} = B_{\alpha}$. And to find $B_{\alpha}$, notice that
\[
B = \frac{pa_p(\phi)}{p+1}(p^k-p^{k-1}+a_p -A) = \frac{pa_p(\phi)}{p+1}(\alpha_f-p^{k-1} - (\alpha_f-p^{k-1})A_{\alpha}),
\]
and therefore 
\[
B_{\alpha} = \frac{pa_p(\phi)}{p+1}(1-A_{\alpha}).
\]
Finally, one can proceed similarly for $D_{\alpha}$. Indeed, we have 
\begin{align*}
D + \beta_fp^k & = p^{k+1}(A-a_p-p^k) + \beta_f p^k = p^{k+1}((\alpha_f-p^{k-1})A_{\alpha} - \alpha_f) + \beta_f p^k = \\
& = (\alpha_f-p^{k-1})p^{k+1}A_{\alpha} -p^{k+1}\alpha_f + \beta_f p^k = (\alpha_f-p^{k-1})p^{k+1}(A_{\alpha}-1) - (\alpha_f-p^{k-1})p \beta_f,
\end{align*}
which implies that
\[
D_{\alpha} = p^{k+1}(A_{\alpha}-1) - p \beta_f.
\]
\end{proof}

Finally, with the same notation as above, in the proof of Proposition \ref{prop:JggF} we have omitted part of the computation of 
\[
  \frac{\langle \mathrm{e}_\mathrm{ord} (\varpi_{\phi}(F_\alpha)), \phi_\alpha\times \phi_\alpha \rangle}{\langle \phi_\alpha, \phi_\alpha\rangle^2},
\]
which equals the coefficient of $\phi_{\alpha}\times \phi_{\alpha}$ when expressing  $\varpi_{\phi}(F_{\alpha})$ in terms of the basis $\phi_{\alpha} \times \phi_{\alpha}$, $\phi_{\alpha} \times \phi_{\beta}$, $\phi_{\beta} \times \phi_{\alpha}$, $\phi_{\beta} \times \phi_{\beta}$. We complete this in the corollary below.

\begin{corollary} \label{cor:Ordinary Part-appendix}
    Let $\phi\in S_{k+1}(N)$ be a normalised eigenform. Then 
    \[
        \frac{\langle \mathrm{e}_\mathrm{ord} (\varpi_{\phi}(F_\alpha)), \phi_\alpha\times \phi_\alpha \rangle}{\langle \phi_\alpha, \phi_\alpha\rangle^2}
        = 
        \frac{\mathcal E^{\circ}(f,\Ad(\phi))}{\mathcal E(\Ad(\phi))}
        \cdot \frac{\langle \varpi(F),\phi\times \phi\rangle}{\langle \phi,\phi\rangle^2},
    \]
    where 
    \[
    \mathcal E^{\circ}(f,\Ad(\phi)) := \left(1-\frac{\beta_f}{p^k}\right)\left(1-\frac{\beta_f\beta_{\phi}/\alpha_{\phi}}{p^k}\right), \quad \mathcal E(\Ad(\phi)) := \left(1 - \frac{\beta_{\phi}}{\alpha_{\phi}}\right)\left(1 - \frac{\beta_{\phi}}{p\alpha_{\phi}}\right).
    \]
\end{corollary}
\begin{proof}
    From the definitions of $p$-adic stabilisations we have that
    \[
        \phi = \frac{\alpha_{\phi} \phi_\alpha - \beta_{\phi} \phi_\beta}{\alpha_{\phi}-\beta_{\phi}}, 
        \qquad 
        V\phi = \frac{ \phi_\alpha -  \phi_\beta}{\alpha_{\phi}-\beta_{\phi}}.
    \]
    In other words, the matrix 
    \[
    \frac{1}{\alpha_{\phi}-\beta_{\phi}}\begin{pmatrix} \alpha_{\phi} & 1 \\ -\beta_{\phi} & -1\end{pmatrix}
    \]
    gives the change of basis from $\{\phi,V\phi\}$ to $\{\phi_{\alpha},\phi_{\beta}\}$ on $S_{k+1}(Np)$. Taking the tensor product of this matrix with itself, it follows that 
    \[
    \frac{1}{(\alpha_{\phi}-\beta_{\phi})^2}\begin{pmatrix} \alpha_{\phi}^2 & \alpha_{\phi} & \alpha_{\phi} & 1 \\ -\alpha_{\phi}\beta_{\phi} & -\alpha_{\phi} & -\beta_{\phi} & -1 \\ -\alpha_{\phi}\beta_{\phi} & -\beta_{\phi} & - \alpha_{\phi} & -1 \\ \beta_{\phi}^2 & \beta_{\phi} & \beta_{\phi} & 1\end{pmatrix}
    \]
    gives the change on $S_{k+1}(Np)[\phi]\otimes S_{k+1}(Np)[\phi]$ from the basis $\{\phi\times \phi, \phi\times V\phi, V\phi \times \phi, V\phi\times V\phi\}$ to the basis $\{\phi_{\alpha}\times\phi_{\alpha}, \phi_{\alpha}\times\phi_{\beta},\phi_{\beta}\times\phi_{\alpha},\phi_{\beta}\times\phi_{\beta}\}$. Using the expression for $\varpi_{\phi}(F_{\alpha})$ computed in the previous corollary, we find that
    \begin{align*}
        \mathrm{e}_\mathrm{ord} (\varpi_{\phi}(F_\alpha))
        = 
        \frac{\lambda_\phi\left(1-\frac{\beta_f}{p^k}\right)}{\left( 1-\frac{\beta_{\phi}}{\alpha_{\phi}} \right)^2 } (A_{\alpha}+B_{\alpha}\alpha_{\phi}^{-1}+C_{\alpha}\alpha_{\phi}^{-1}+D_{\alpha}\alpha_{\phi}^{-2}) \phi_\alpha \times \phi_\alpha,
    \end{align*}
    where $\lambda_{\phi} = \frac{\langle \varpi(F),\phi\times \phi\rangle}{\langle \phi,\phi\rangle^2}$. And using the expressions for $A_{\alpha}$, $B_{\alpha}$, $C_{\alpha}$, $D_{\alpha}$ from Corollary \ref{cor:Falpha-appendix} we find that 
    \begin{align*}
        \mathrm{e}_\mathrm{ord} (\varpi_{\phi}(F_\alpha))
        &= 
        \frac{\lambda_\phi\left(1-\frac{\beta_f}{p^k}\right)}{\left( 1-\frac{\beta_{\phi}}{\alpha_{\phi}} \right)^2 } \left( A_{\alpha} - (A_{\alpha}-1)\frac{2p a_p(\phi)}{\alpha(p+1)} + ((A_{\alpha}-1)p^{k+1}-p\beta_f)\frac{1}{\alpha^2} \right) \phi_\alpha \times \phi_\alpha 
        = \\
        & = \frac{\lambda_\phi\left(1-\frac{\beta_f}{p^k}\right)}{\left( 1-\frac{\beta_{\phi}}{\alpha_{\phi}} \right)^2 } \left( (A_{\alpha}-1)\left(1 - \frac{2p a_p(\phi)}{\alpha_{\phi}(p+1)} + \frac{p\beta_{\phi}}{\alpha_{\phi}}\right) + 1 - \frac{p\beta_f}{\alpha^2}\right) \phi_\alpha \times \phi_\alpha =\\
        & = \frac{\lambda_\phi\left(1-\frac{\beta_f}{p^k}\right)}{\left( 1-\frac{\beta_{\phi}}{\alpha_{\phi}} \right)^2 } \left(\frac{(p-1)\left(1-\frac{p\beta_{\phi}}{\alpha_{\phi}}\right)\left(1-\frac{\beta_f}{p^{k-1}}\right)}{\left(p-\frac{\alpha_{\phi}}{\beta_{\phi}}\right)\left(p-\frac{\beta_{\phi}}{\alpha_{\phi}}\right)}+\left(1-\frac{\beta_f\beta_{\phi}/\alpha_{\phi}}{p^{k-1}}\right)\right) \phi_\alpha\times\phi_\alpha,
    \end{align*}
    where we have used that
    \begin{align*}
    1-\frac{2pa_p(\phi)}{\alpha_{\phi}(p+1)}+\frac{p\beta_{\phi}}{\alpha_{\phi}} & = \frac{p\alpha_{\phi}+\alpha_{\phi}-2p\alpha_{\phi}-2p\beta_{\phi}+p^2\beta_{\phi}+p\beta_{\phi}}{\alpha_{\phi}(p+1)} = \frac{\alpha_{\phi}-p\alpha_{\phi}-p\beta_{\phi}+p^2\beta_{\phi}}{\alpha_{\phi}(p+1)}= \\
    & = \frac{1-p-p\beta_{\phi}/\alpha_{\phi}+p^2\beta_{\phi}/\alpha_{\phi}}{p+1} = -\frac{(p-1)\left(1-\frac{p\beta_{\phi}}{\alpha_{\phi}}\right)}{p+1}
    \end{align*}
    and that 
    \[
    A_{\alpha}-1 = -  \frac{(p+1)\left(1-\frac{\beta_f}{p^{k-1}}\right)}{\left(p-\frac{\alpha_{\phi}}{\beta_{\phi}}\right)\left(p-\frac{\beta_{\phi}}{\alpha_{\phi}}\right)}.
    \]
After some elementary algebra, the above computation eventually yields
\begin{align*}
    \frac{\langle \mathrm{e}_\mathrm{ord} (\varpi_{\phi}(F_\alpha)), \phi_{\alpha}\times\phi_{\alpha}\rangle}{\langle \phi_{\alpha}\times\phi_{\alpha},\phi_{\alpha}\times\phi_{\alpha}\rangle} & =\frac{p\left(1-\frac{\beta_f}{p^k}\right)\left(1-\frac{\beta_f\beta_{\phi}/\alpha_{\phi}}{p^k}\right)\left(p-\frac{\alpha_{\phi}}{\beta_{\phi}}\right)\left(1-\frac{\beta_{\phi}}{\alpha_{\phi}}\right)}{\left(p-\frac{\alpha_{\phi}}{\beta_{\phi}}\right)\left(p-\frac{\beta_{\phi}}{\alpha_{\phi}}\right)\left(1-\frac{\beta_{\phi}}{\alpha_{\phi}}\right)^2} \lambda_{\phi} 
    = \\
    & = \frac{p\left(1-\frac{\beta_f}{p^k}\right)\left(1-\frac{\beta_f\beta_{\phi}/\alpha_{\phi}}{p^k}\right)}{\left(p-\frac{\beta_{\phi}}{\alpha_{\phi}}\right)\left(1-\frac{\beta_{\phi}}{\alpha_{\phi}}\right)} \lambda_{\phi} = \frac{\left(1-\frac{\beta_f}{p^k}\right)\left(1-\frac{\beta_f\beta_{\phi}/\alpha_{\phi}}{p^k}\right)}{\left(1-\frac{\beta_{\phi}}{\alpha_{\phi}}\right)\left(1-\frac{\beta_{\phi}}{p\alpha_{\phi}}\right)} \lambda_{\phi},
\end{align*}
thereby proving the result.
\end{proof}


\end{document}